\documentclass[Unicode]{cedram-alco}
% ...
\title
[Patterns and polytopes]
{The feasible region for consecutive patterns of permutations is a cycle polytope}

\author[\initial{J.} Borga]{\firstname{Jacopo} \lastname{Borga}}

\address{Institut für Mathematik, Universität Zürich, Winterthurerstr. 190, CH-8057 Zürich, Switzerland}

\email{jacopo.borga@math.uzh.ch}

\urladdr{http://www.jacopoborga.com/}

\author[\initial{R.} Penaguiao]{\firstname{Raul} \lastname{Penaguiao}}
\address{Institut für Mathematik, Universität Zürich, Winterthurerstr. 190, CH-8057 Zürich, Switzerland}
\email{raul.penaguiao@math.uzh.ch}

\thanks{This work was completed with the support of the SNF grants number 200021-172536 and 200020-172515.}



\usepackage{tikz}
\usepackage{bm}
\usepackage{blkarray, bigstrut}

\usepackage[capitalize]{cleveref}
\equalenv{definition}{defi}
\equalenv{proposition}{prop}
\equalenv{remark}{rema}
\equalenv{example}{exam}

\def\CCC{\mathcal{C}}  %
\def\SS{\mathcal{S}}   %
\def\NN{\mathbb N}

\def\R{\mathbb{R}}

\def\sc{\mathcal{C}}

\DeclareMathOperator{\occ}{occ}

\DeclareMathOperator{\pat}{pat}

\DeclareMathOperator{\be}{beg}
\DeclareMathOperator{\en}{end}
\DeclareMathOperator{\st}{s}
\DeclareMathOperator{\ar}{a}
\DeclareMathOperator{\Aff}{Aff}

\DeclareMathOperator{\vol}{vol}
\DeclareMathOperator{\conv}{conv}

\DeclareMathOperator{\cnv}{conv} %
\DeclareMathOperator{\dist}{dist} %
\DeclareMathOperator{\lb}{lb} %
\DeclareMathOperator{\std}{std} %
\DeclareMathOperator{\spn}{span} %
\DeclareMathOperator{\coc}{c-occ} %
\DeclareMathOperator{\oc}{occ} %

\def\Z{\mathbb{Z}}

\def\pcoc{\widetilde{\coc}}
\def\poc{\widetilde{\oc}}
\def\JacGraph[#1]{G_J(#1)}
\def\ValGraph[#1]{\mathcal{O}v(#1)}
\def\vecc[#1]{\vec{e}_{\CCC_{#1}}}

\newcounter{indice}

\newcommand{\permutation}[1]{
	\setcounter{indice}{0};
	\foreach \i in {#1}
	\addtocounter{indice}{1};
	
	\addtocounter{indice}{1}
	\draw [help lines] (1,1) grid (\theindice,\theindice);
	
	\setcounter{indice}{1};
	
	\foreach \i in { #1 } {
		\draw (\theindice+.5,\i+.5) [fill] circle (.2);
		\addtocounter{indice}{1};
	}
	\addtocounter{indice}{-1};
	
}

\begin{document}
\newtheorem{observation}[cdrthm]{Observation}
\crefname{Observation}{Observation}{Observations}

\begin{abstract}
We study proportions of consecutive occurrences of permutations of a given size. Specifically, the limit of such proportions on large permutations forms a region, called \emph{feasible region}. We show that this feasible region is a polytope, more precisely the cycle polytope of a specific graph called \emph{overlap graph}. This allows us to compute the dimension, vertices and faces of the polytope, and to determine the equations that define it. 
Finally we prove that the limit of classical occurrences and consecutive occurrences are in some sense independent. As a consequence, the scaling limit of a sequence of permutations induces no constraints on the local limit and vice versa.
\end{abstract}

\maketitle

\begin{figure}[h]
	\centering
	\includegraphics[scale=0.73]{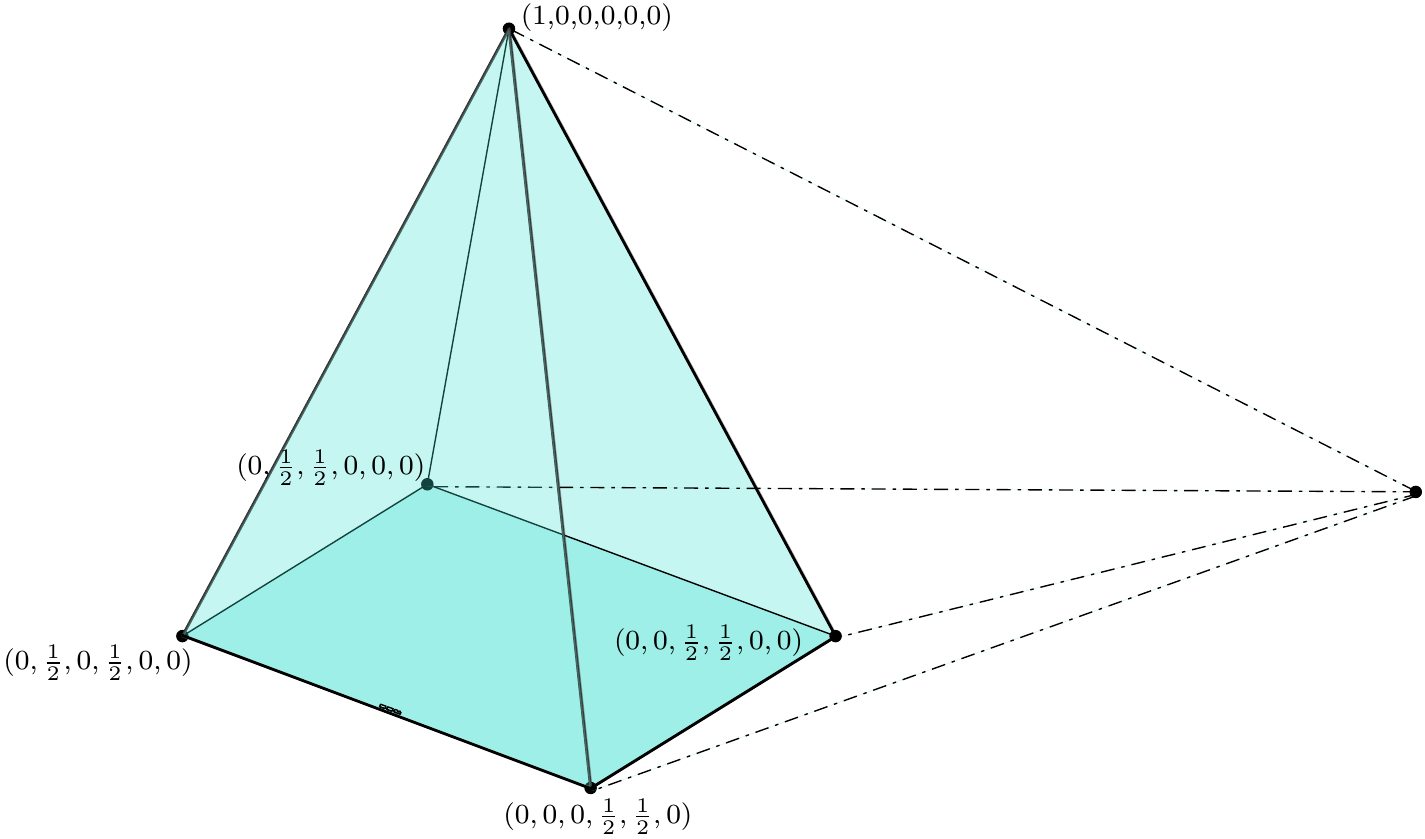}
	\bigskip
	\caption{The four-dimensional polytope $P_3$ given by the six patterns of size three (see \cref{eq:setofinter} for a precise definition). We highlight in light-blue one of the six three-dimensional facets of $P_3$. This facet is a pyramid with square base. The polytope itself is a four-dimensional pyramid, whose base is the highlighted facet. \label{fig:P_3}}
\end{figure}

\newpage

\section{Introduction}

\label{sect:intro}

\subsection{Motivations}

\label{sect:occ and cocc}

Despite this article not containing any probabilistic result, we introduce here some motivations that come from the study of random permutations.
This is a classical topic at the interface of combinatorics and discrete probability theory. There are two main approaches to the topic: the first concerns the study of statistics on permutations, and the second, more recent, looks for limits of sequences of permutations. The two approaches are not orthogonal and many results relate them, for instance Theorem \ref{thm:carac_occ} and Theorem \ref{thm:carac_cocc} below.  

In order to study the shape of permutations, two main notions of convergence have been defined: a global notion of convergence (called permuton convergence) and a local notion of convergence (called Benjamini--Schramm convergence, or BS-convergence, for short).
The notion of permuton limit for permutations has been introduced  in \cite{MR2995721}. 
A permuton is a probability measure on the unit square with uniform marginals, and represents the scaling limit of a sequence of permutations seen as permutation matrices, as the size grows to infinity. The study of permuton limits is an active and exciting research field in combinatorics, see for instance \cite{bassino2017universal,bassino2019scaling,MR3813988,borga2019decorated,borga2019almost, kenyon2015permutations,rahman2016geometry,MR2266895,MR2566888}. 
On the other hand, the notion of BS-limit for permutations is more recent, having been introduced in \cite{borga2018local}. 
Informally, in order to investigate BS-limits, we study the permutation in a neighborhood around a randomly marked point. Limiting objects for this framework are called \emph{infinite rooted permutations} and are in bijection with total orders on the set of integer numbers. 
BS-limits have also been studied in some other works, see for instance \cite{bevan2019permutations,borga2019decorated,borga2019square}.

We denote by $\mathcal{S}_n$ the set of permutations of size $n$, by $\SS$ the space of all permutations, and by $\widetilde{\occ}(\pi,\sigma)$ (resp.\ $\pcoc(\pi,\sigma)$) the proportion of classical occurrences (resp.\ consecutive occurrences) of a permutation $\pi$ in $\sigma$ (see \cref{sect:notation} for notation and basic definitions).
The following theorems provide relevant combinatorial characterizations of the two aforementioned notions of convergence.

\begin{theorem}[\cite{MR2995721}]\label{thm:carac_occ}
	For any $n\in\NN$, let $\sigma^n\in\SS$ and assume that $|\sigma^n|\to\infty$. The sequence $(\sigma^n)_{n\in\NN}$ converges to some limiting permuton $P$ if and only if there exists a vector $(\Delta_{\pi}(P))_{\pi\in\mathcal{S}}$ of non-negative real numbers (that depends on $P$) such that, for all $\pi\in\mathcal{S}$,
	$$\widetilde{\occ}(\pi,\sigma^n)\to\Delta_{\pi}(P).$$
\end{theorem}

\begin{theorem}[\cite{borga2018local}]\label{thm:carac_cocc}
	For any $n\in\NN$, let $\sigma^n\in\SS$ and assume that $|\sigma^n|\to\infty$. The sequence $(\sigma^n)_{n\in\NN}$ converges in the Benjamini--Schramm topology to some random infinite rooted permutation $\sigma^\infty$ if and only if there exists a vector $(\Gamma_{\pi}(\sigma^\infty))_{\pi\in\mathcal{S}}$ of non-negative real numbers (that depends on $\sigma^\infty$) such that, for all $\pi\in\mathcal{S}$,
	$$\widetilde{\coc}(\pi,\sigma^n)\to\Gamma_{\pi}(\sigma^\infty).$$
\end{theorem}

A natural question, motivated by the theorems above, is the following: given a finite family of patterns $\mathcal{A}\subseteq\SS$ and a vector $(\Delta_\pi)_{\pi\in\mathcal{A}}\in [0,1]^{\mathcal A}$, or $(\Gamma_\pi)_{\pi\in\mathcal{A}}\in [0,1]^{\mathcal A}$, does there exist a sequence of permutations $(\sigma^n)_{n\in\NN}$ such that $|\sigma^n|\to\infty$ and
$$\widetilde{\occ}(\pi,\sigma^n)\to\Delta_{\pi}, \quad \text{for all} \quad \pi\in\mathcal{A},$$
or
$$\widetilde{\coc}(\pi,\sigma^n)\to\Gamma_{\pi}, \quad \text{for all} \quad \pi\in\mathcal{A}\;?$$

We consider the classical pattern limiting sets, sometimes called the \textit{feasible region} for (classical) patterns, defined as
\begin{align}\label{eq:setofnotinter}
clP_k \coloneqq&\left\{\vec{v}\in [0,1]^{\SS_k} \big| \exists (\sigma^m)_{m\in\NN} \in \SS^{\NN} \text{ s.t. }|\sigma^m| \to \infty\text{ and } \poc(\pi, \sigma^m ) \to \vec{v}_{\pi},\forall \pi\in\SS_k  \right\}\\
=&\left\{(\Delta_{\pi}(P))_{\pi\in\SS_k} \big| P\text{ is a permuton}  \right\} \, ,\nonumber
\end{align}
and we	 introduce the consecutive pattern limiting sets, called here the \textit{feasible region} for consecutive patterns,
\begin{align}
\label{eq:setofinter}
P_k \coloneqq &\left\{\vec{v}\in [0,1]^{\SS_k} \big| \exists (\sigma^m)_{m\in\NN} \in \SS^{\NN} \text{ s.t. }|\sigma^m| \to
\infty \text{ and }  \pcoc(\pi, \sigma^m ) \to \vec{v}_{\pi}, \forall \pi\in\SS_k \right\}\\
=&\left\{(\Gamma_{\pi}(\sigma^{\infty}))_{\pi\in\SS_k} \big| \sigma^{\infty}\text{ is a random infinite rooted \emph{shift-invariant} permutation}  \right\} \, . \nonumber
\end{align}

We present the definition of \emph{shift-invariant} permutation in Definition \ref{def_equality_intro}, and we prove the equality in \cref{eq:setofinter} in Proposition \ref{cor_equality_intro}. The equality in \cref{eq:setofnotinter} follow from \cite[Theorem 1.6]{MR2995721}.

The feasible region $clP_k$ was previously studied in several papers (see \cref{sect:feas_reg}). 
The main goal of this project is to analyze the feasible region $P_k$, that turns out to be connected to specific graphs called \emph{overlap graphs} (see \cref{sect:ov_gr}) and their corresponding cycle polytopes (see \cref{sect:cyc_pol}).

\subsection{The feasible region for classical patterns}
\label{sect:feas_reg}

The feasible region $clP_k$ was first studied in \cite{kenyon2015permutations} for some particular families of patterns instead of the whole $\SS_k$. 
More precisely, given a list of finite sets of permutations $(\mathcal{P}_1,\dots,\mathcal{P}_\ell)$, the authors considered the \emph{feasible region} for $(\mathcal{P}_1,\dots,\mathcal{P}_\ell)$, that is, the set
$$\left\{\vec{v}\in [0,1]^{\ell} \Bigg| \exists (\sigma^m)_{m\in\NN} \in \SS^{\NN} \text{ s.t. }|\sigma^m| \to \infty\text{ and } \sum_{\tau\in\mathcal{P}_i}\poc(\tau,\sigma^m ) \to \vec{v}_i,\text{ for }i\in[\ell]  \right\} \, .$$

They first studied the simplest case when $\mathcal{P}_1=\{12\}$ and $\mathcal{P}_2=\{123,213\}$ showing that the corresponding feasible region for $(\mathcal{P}_1 ,\mathcal{P}_2)$ is the region of the square $[0,1]^2$ bounded from below by the parameterized curve $(2t-t^2,3t^2-2t^3)_{t\in[0,1]}$ and from above by the parameterized curve $(1-t^2,1-t^3)_{t\in[0,1]}$  (see \cite[Theorem 13]{kenyon2015permutations}).

They also proved in \cite[Theorem 14]{kenyon2015permutations} that if each $\mathcal{P}_i = \{\tau_i \}$ is a singleton, and there is some value $p$ such that, for all permutations $\tau_i$, the final element $\tau_i(|\tau_i|)$ is equal to $p$ , then the corresponding feasible region is convex. 
They remarked that one can construct examples where the feasible region is not strictly convex: e.g.\ in the case where $\mathcal{P}_1 = \{231,321\}$ and $\mathcal{P}_2 = \{123,213\}$.

They finally studied two additional examples: the feasible regions for $(\{12\},\{123\})$ (see \cite[Theorem 15]{kenyon2015permutations}) and for the patterns $(\{123\},\{321\})$ (see \cite[Section 10]{kenyon2015permutations}). In the first case, they showed that the feasible region is equal to the so-called “scalloped triangle” of
Razborov \cite{MR2433944,MR2371204} (this region also describes the space of limit densities for edges and triangles in graphs). 
For the second case, they showed that the feasible region is equal to the limit of densities of triangles versus the density of anti-triangles in graphs, see \cite{MR3200287,MR3572422}.

\medskip

The set $clP_k$ was also studied in \cite{MR3567538}, even though with a different goal.
There, it was shown that $clP_k$ contains an open ball $B$ with dimension $|I_k|$, where $I_k$ is the set of $\oplus$-indecomposable permutations of size at most $k$.
Specifically, for a particular ball $B\subseteq \R^{I_k}$, the authors constructed permutons $P_{\vec{x}}$ such that $\Delta_{\pi }(P_{\vec{x}}) = \vec{x}_{\pi}$, for each point $\vec{x} \in B$.

This work opened the problem of finding the maximal dimension of an open ball contained in $clP_k$, and placed a lower bound on it.
In \cite{vargas2014hopf} an upper bound for this maximal dimension was indirectly given as the number of so-called \textit{Lyndon permutations} of size at most $k$, whose set we denote $\mathcal{L}_k$.
In this article, the author showed that for any permutation $\pi$ that is not a Lyndon permutation, $\poc(\pi, \sigma ) $ can be expressed as a polynomial on the functions $\{\poc(\tau, \sigma ) |\tau \in \mathcal{L}_k \}$ that does not depend on $\sigma$.
It follows that $clP_k$ sits inside an algebraic variety of dimension $|\mathcal{L}_k|$.
We expect that this bound is sharp since, from our computations, this is the case for small values of $k$.

\begin{conjecture}
The feasible region $clP_k$ is full-dimensional inside a manifold of dimension $|\mathcal{L}_k|$.
\end{conjecture}

\subsection{Main result}
Unlike with the case of classical patterns, we are able to obtain here a full description of the feasible region $P_k$ as the cycle polytope of a specific graph, called the \emph{overlap graph} $\ValGraph[k]$. 

\begin{definition}
	The graph $\ValGraph[k]$ is a directed multigraph with labeled edges, where the vertices are elements of $\SS_{k-1}$ and for every $\pi\in\SS_{k}$ there is an edge labeled by $\pi$ from the pattern induced by the first $k-1$ indices of $\pi$ to the pattern induced by the last $k-1$ indices of $\pi$.
\end{definition}

The overlap graph $\ValGraph[4]$ is displayed in \cref{Overlap_graph_exemp}.

\begin{figure}[htbp]
	\begin{center}
		\includegraphics[scale=.8]{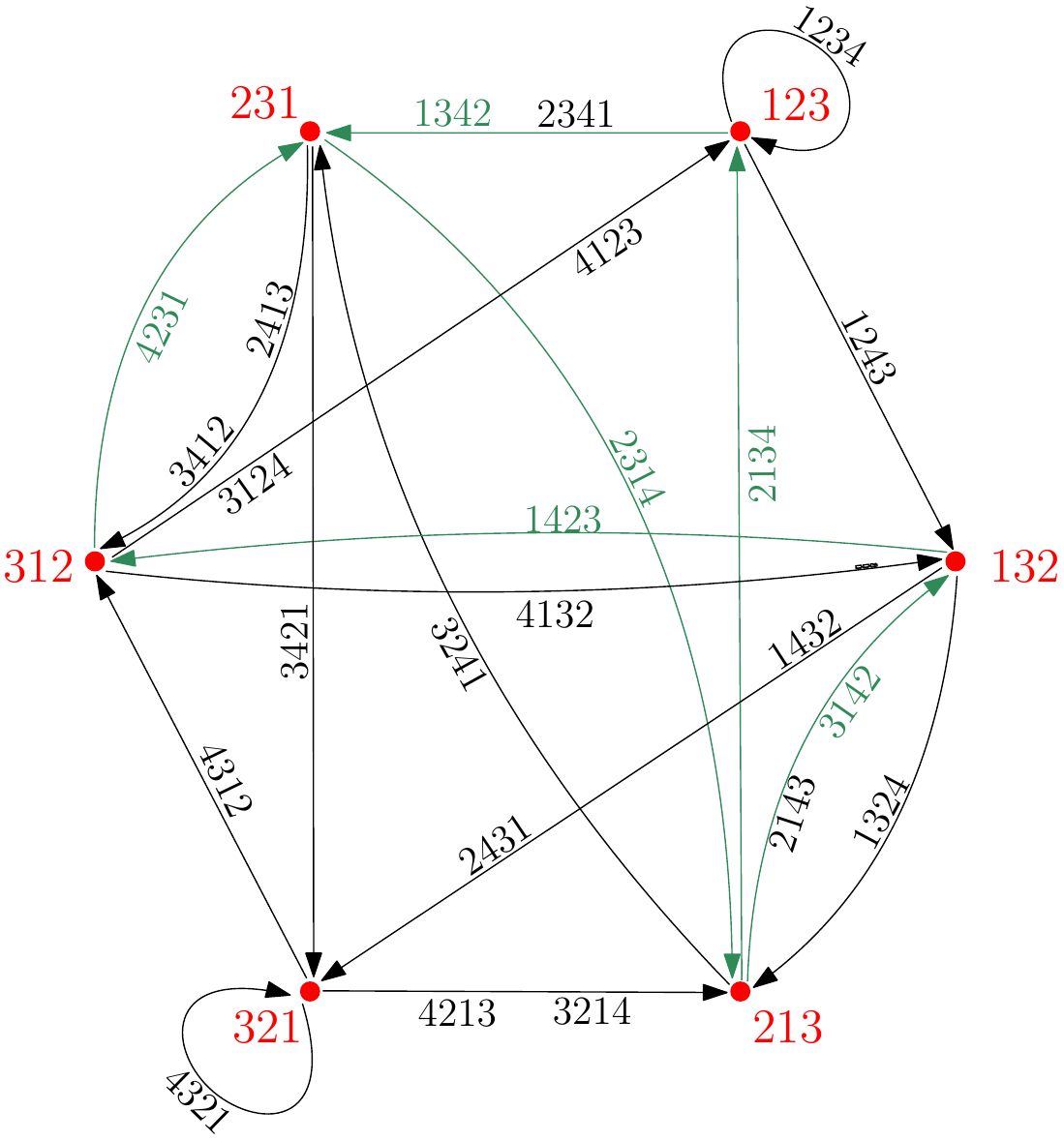}\\
		\caption{The overlap graph $\ValGraph[4]$. The six vertices are painted in red and the edges are drawn as labeled arrows. Note that in order to obtain a clearer picture we did not draw multiple edges, but we use multiple labels (for example the edge $231 \to 312$ is labeled with the permutations $3412$ and $2413$ and should be thought of as two distinct edges labeled with $3412$ and $2413$ respectively). The role of the green arrows is clarified in Example \ref{exemp:overlap_graph2}. \label{Overlap_graph_exemp}}
	\end{center}
\end{figure}

\begin{definition}
	Let $G=(V,E)$ be a directed multigraph.
	For each non-empty cycle $\mathcal{C}$ in $G$, define $\vec{e}_{\mathcal{C}}\in \mathbb{R}^{E}$ so that
	$$(\vec{e}_{\mathcal{C}})_e \coloneqq \frac{\text{\# of occurrences of $e$ in $\sc$}}{|\mathcal{C}|}, \quad \text{for all}\quad e\in E. $$
	We define $P(G) \coloneqq \conv \{\vec{e}_{\mathcal{C}} | \, \mathcal{C} \text{ is a simple cycle of } G \}$, the \emph{cycle polytope} of $G$.
\end{definition}

Our main result is the following.

\begin{theorem}
	\label{thm:main_res}
	$P_k$ is the cycle polytope of the overlap graph $\ValGraph[k]$. Its dimension is $k! - (k-1)!$ and its vertices are given by the simple cycles of $\ValGraph[k]$.
\end{theorem}

In addition, we also determine the equations that describe the polytope $P_k$ (for a precise statement see Theorem \ref{thm:vertices}).

In order to find the dimension, the vertices and the equations describing $P_k$, we first recover and prove general results for cycle polytopes of directed multigraphs (see \cref{sect:cyc_pol}). Then we apply them to the specific case of our graph of interest.

\subsection{The overlap graph}
\label{sect:ov_gr}
Overlap graphs were already studied in previous works. We give here a brief summary of the relevant literature. The overlap graph $\ValGraph[k]$ is the line graph of the \emph{de Bruijn graph for permutations} of size $k-1$. The latter was introduced in \cite{MR1197444}, where the authors studied universal cycles (sometime also called \emph{de Bruijn cycles}) of several combinatorial structures, including permutations. 
In this case, a universal cycle of order $n$ is a cyclic word of size $n!$ on an alphabet of $N$ letters that contains all the patterns of size $n$ as consecutive patterns. In \cite{MR1197444} it was conjectured (and then proved in \cite{MR2548540}) that such universal cycles always exist when the alphabet is of size $N=n+1$. 

The \emph{de Bruijn graph for permutations} was also studied under the name of \emph{graph of overlapping permutations} in \cite{MR3944621} again in relation with universal cycles and universal words. Further, in \cite{asplund2018enumerating} the authors enumerate some families of cycles of these graphs.

We mainly use overlap graphs as a tool to state and prove our results on $P_k$, rather than exploiting its properties. 
We remark that, applying the same ideas used to show the existence of Eulerian and Hamiltonian cycles in classical de Bruijn graphs (see \cite{MR1197444}), it is easy to prove the existence of both Eulerian and Hamiltonian cycles in $\ValGraph[k]$.
In particular, with an Eulerian path in $\ValGraph[k]$, we can construct (although not uniquely) a permutation $\sigma $ of size $k!+k -1$ such that $\coc (\pi,  \sigma ) = 1 $ for any $\pi\in \SS_k$.

\subsection{Polytopes and cycle polytopes}
\label{sect:cyc_pol}

As said before, we recover and prove some general results for cycle polytopes of directed multigraphs. A first result is the following consequence of \cite[Proposition 6]{gleiss2001circuit} (see Proposition \ref{thm:dimension} below for a proof).

\begin{proposition}
	\label{thm:dim_cyc_pol}
	The cycle polytope of a strongly connected directed multigraph $G=(V,E)$ has dimension $|E|-|V|$.
\end{proposition}

We also recover from \cite[Proposition 1]{gleiss2001circuit} the equations defining the polytope (see Proposition \ref{thm:faces}) and we show that all its faces can be identified with some subgraphs of $G$ (see Theorem \ref{cor:facestruct}). This gives us a description of the face poset of the polytope.
Further, the computation of the dimension is generalized for any cycle polytope, even those that do not come from strongly connected graphs (see Proposition \ref{thm:dimfullgr}).

\subsection{Mixing classical patterns and consecutive patterns}
\label{sect:mixing}
We saw in \cref{sect:feas_reg} that the feasible region $clP_k$ for classical pattern occurrences has been studied in several papers. In this paper we study the feasible region $P_k$ of limiting points for consecutive pattern occurrences. A natural question is the following: what is the feasible region if we mix classical and consecutive patterns?

We answer this question showing that:
\begin{theorem}
	\label{thm:mixing}
	For any two points $\vec{v}_1 \in clP_k, \vec{v}_2\in P_k$, there exists a sequence of permutations $(\sigma^m)_{m\in\NN}$ such that $|\sigma^m| \to \infty$, satisfying  
	$$(\poc(\pi, \sigma^m ))_{\pi\in\SS_k} \to \vec{v}_1 \quad \text{ and } \quad (\pcoc(\pi, \sigma^m ))_{\pi\in\SS_k} \to \vec{v}_2.$$
\end{theorem}

This result shows a sort of independence between classical patterns and consecutive patterns, in the sense that knowing the proportion of classical patterns of a certain sequence of permutations imposes no constraints for the proportion of consecutive patterns on the same sequence and vice versa.

We stress that we provide an explicit construction of the sequence $(\sigma^m)_{m\in\NN}$ in the theorem above (for a more precise and general statement, see Theorem \ref{thm:ABC}).

We conclude this section with the following observation on local and scaling limits of permutations.
\begin{observation}
	In Theorem \ref{thm:carac_occ} and Theorem \ref{thm:carac_cocc} we saw that the proportion of occurrences (resp.\ consecutive occurrences) in a sequence of permutations $(\sigma^m)_{m\in\NN}$ characterizes the permuton limit (resp.\ Benjamini--Schramm limit) of the sequence. Theorem \ref{thm:mixing} proves that the permuton limit of a sequence of permutations induces no constraints for the Benjamini--Schramm limit and vice versa. For instance, we can construct a sequence of permutations where the permuton limit is the decreasing diagonal and the Benjamini--Schramm limit is the classical increasing total order on the integer numbers.
	
	We remark that a particular instance of this ``independence phenomenon'' for local/scaling limits of permutations was recently also observed by Bevan, who pointed out in the abstract of \cite{bevan2019permutations} that ``the knowledge of the local structure of uniformly random permutations with a specific fixed proportion of inversions reveals nothing about their global form''. Here, we prove that this is a \emph{universal phenomenon} which is not specific to the framework studied by Bevan.
\end{observation}

\subsection{Outline of the paper} 
\label{sect:main results}

The paper is organized as follows:
\begin{itemize}
	\item In \cref{sec:Cycle_pol_dim} we analyze directed multigraphs and consider their \textit{cycle polytopes}. 
	There, we prove Proposition \ref{thm:dim_cyc_pol} and the results mentioned immediately below it.
	
	\item Our results regarding $P_k$ come in \cref{sec:Pn_is_poly}, where we prove Theorem \ref{thm:main_res}.
	
	\item Finally, we prove in \cref{sec:class_and_cons} a more precise version of Theorem \ref{thm:mixing}. 
\end{itemize}

\subsection{Notation}
\label{sect:notation} 

We summarize here the notation and some basic definitions used in the paper.

\subsubsection*{Permutations and patterns}
We let $\NN=\{1,2,\dots\}$ denote the collection of strictly positive integers. For every $n\in\NN$, we view permutations of $[n]=\{1,2,\dots,n\}$ as words of size $n$, and write them using the one-line notation $\sigma=\sigma(1)\sigma(2)\dots\sigma(n)$. We denote by $\mathcal{S}_n$ the set of permutations of size $n$, by $\mathcal{S}_{\geq n}$ the set of permutations of size at least $n$, and by $\mathcal{S}$ the set of permutations of finite size.  

We often view a permutation $\sigma\in\SS_n$ as a diagram, specifically as an $n\times n$ board with $n$ points at positions $(i,\sigma(i))$ for all $i\leq n$.

If $x_1, \dots , x_n$ is a sequence of distinct numbers, let $\std(x_1, \dots,  x_n)$ be the unique permutation $\pi$ in $\SS_n$ whose elements are in the same relative order as $x_1, \dots,  x_n$, i.e.\ $\pi(i)<\pi(j)$ if and only if $x_i<x_j.$
Given a permutation $\sigma\in\SS_n$ and a subset of indices $I\subseteq[n]$, let $\text{pat}_I(\sigma)$ be the permutation induced by $(\sigma(i))_{i\in I}$, namely, $\text{pat}_I(\sigma)\coloneqq\std\left((\sigma(i))_{i\in I}\right).$
For example, if $\sigma=87532461$ and $I=\{2,4,7\}$, then $\text{pat}_{\{2,4,7\}}(87532461)=\std(736)=312$. 

Given two permutations, $\sigma\in\SS_n$, $\pi\in\SS_k$ for some positive integers $n \geq k$, we say that $\sigma$ contains $\pi$ as a \textit{pattern} if there exists a \emph{subset} $I\subseteq[n]$ such that $\text{pat}_I(\sigma)=\pi$, that is, if $\sigma$ has a subsequence of entries order-isomorphic to $\pi$.
Denoting by $i_1, i_2, \dots, i_k$ the elements of $I$ in increasing order, the subsequence $\sigma(i_1) \sigma(i_2) \dots \sigma(i_k)$ is called an \emph{occurrence} of $\pi$ in $\sigma$. 
In addition, we say that $\sigma$ contains $\pi$ as a \textit{consecutive pattern} if there exists an \emph{interval} $I\subseteq[n]$ such that $\text{pat}_I(\sigma)=\pi$, 
that is, if $\sigma$ has a subsequence of adjacent entries order-isomorphic to $\pi$. 
Using the same notation as above, $\sigma(i_1) \sigma(i_2) \dots \sigma(i_k)$ is then called a \emph{consecutive occurrence} of $\pi$ in $\sigma$. 

We denote by $\oc(\pi,\sigma)$ the number of occurrences of a pattern $\pi$ in $\sigma$, more precisely
\begin{equation*}
\oc(\pi,\sigma)\coloneqq \Big| \Big\{I\subseteq[n]| \text{pat}_I(\sigma)=\pi\Big\}\Big| \, .
\end{equation*}
We denote by $\coc(\pi,\sigma)$ the number of consecutive occurrences of a pattern $\pi$ in $\sigma$, more precisely
\begin{equation*}
\coc(\pi,\sigma)\coloneqq \Big|\Big\{I\subseteq[n] | \;I\text{ is an interval, } \text{pat}_I(\sigma)=\pi\Big\} \Big|.
\end{equation*}
Moreover, we denote by $\widetilde{\oc}(\pi,\sigma)$ (resp.\ by $\widetilde{\coc}(\pi,\sigma)$) the proportion of occurrences (resp.\ consecutive occurrences) of a pattern $\pi$ in $\sigma$, that is,
\begin{equation*}
\label{conpatden}
\poc(\pi,\sigma)\coloneqq\frac{\oc(\pi,\sigma)}{\binom{n}{k}}\in[0,1] , \quad \quad \pcoc(\pi,\sigma)\coloneqq\frac{\coc(\pi,\sigma)}{n}\in[0,1]\, .
\end{equation*}
\begin{remark}
The natural choice for the denominator of the expression in the right-hand side of the equation above should be $n-k+1$ and not $n,$ but we make this choice for later convenience. Moreover, for every fixed $k,$ there are no difference in the asymptotics when $n$ tends to infinity. 
\end{remark}
For a fixed $k\in\NN$ and a permutation $\sigma\in\SS_{\geq k}$, we let $\poc_k ( \sigma ), \pcoc_k ( \sigma ) \in [0, 1]^{\SS_k}$ be the vectors 
\[\poc_k ( \sigma )\coloneqq \left(\poc(\pi,\sigma)\right)_{\pi\in\SS_k}, \quad \quad \pcoc_k ( \sigma )\coloneqq \left(\pcoc(\pi,\sigma)\right)_{\pi\in\SS_k}\, .\]

Finally, we denote with $\oplus$ the direct sum of two permutations, i.e.\ for $\tau\in\SS_m$ and $\sigma\in\mathcal{S}_n$, 
$$\tau\oplus\sigma=\tau(1)\dots\tau(k)(\sigma(1)+m)\dots(\sigma(n)+m)\, , $$
and we denote with $\oplus_{\ell}\,\sigma$ the direct sum of $\ell$ copies of $\sigma$ (we remark that the operation $\oplus$ is associative).

\subsubsection*{Polytopes} 

Given a set $S \subseteq \R^n $, we define the \textit{convex hull} (resp.\ \textit{affine span}, \textit{linear span}) of $S$ as the set of all \textit{convex combinations} (resp.\ \textit{affine combinations}, \textit{linear combinations}) of points in $S$, that is

$$ \conv (S) =  \left\{ \sum_{i=1}^k \alpha_i \vec{v}_i \, \Bigg| \vec{v}_i \in S, \alpha_i\in [0, 1] \text{ for } i=1, \dots k \, , \sum_{i = 1}^k \alpha_i = 1  \right\} \, , $$

$$ \Aff (S) =  \left\{ \sum_{i=1}^k \alpha_i \vec{v}_i \, \Bigg| \vec{v}_i \in S, \alpha_i\in \R \text{ for } i=1, \dots k \, , \sum_{i = 1}^k \alpha_i = 1  \right\} \, , $$

$$ \spn (S) =  \left\{ \sum_{i=1}^k \alpha_i \vec{v}_i \, \Bigg| \vec{v}_i \in S, \alpha_i\in \R \text{ for } i=1, \dots k \right\} \, . $$

\begin{definition}[Polytope]
	A polytope $\mathfrak{p}$ in $\R^n$ is a bounded subset of $\R^n$ described by $m$ linear inequalities.
	That is, there is some $m \times n$ real matrix $A$ and a vector $b\in\R^m$ such that
	$$\mathfrak{p} = \{\vec{x} \in \R^n | A\vec{x} \geq b \} \, . $$
	
	The dimension of a polytope $\mathfrak{p}$ in $\R^n$ is the dimension of $\Aff \mathfrak{p}$ as an affine space.
\end{definition}

For any polytope, there is a unique minimal finite set of points $\mathcal{P} \subset \R^n$ such that $\mathfrak{p} = \conv \mathcal{P}$, see \cite{ziegler2012lectures}.
This family $\mathcal{P}$ is called the set of \textit{vertices} of $\mathfrak{p}$.

\begin{definition}[Faces of a polytope]
	Let $\mathfrak{p}$ be a polytope in $\R^n$.
	A linear form in $\R^n$ is a linear map $f:\R^n \to \R$. Its minimizing set on $\mathfrak{p}$ is the subset $\mathfrak{p}_f \subseteq \mathfrak{p}$ where $f$ takes minimal values.
	This set always exists because $\mathfrak{p}$ is compact.
	
	A face of $\mathfrak{p} $ is a subset $\mathfrak{f}\subseteq \mathfrak{p}$ for which there exists a linear form $f$ that satisfies $\mathfrak{p}_f=\mathfrak{f}$.
	A face is also a polytope, and any vertex of $\mathfrak{p}$ is a face of $\mathfrak{p}$.
	The faces of a polytope $\mathfrak{p}$ form a poset when ordered by inclusion, called the \textit{face poset}.
\end{definition}

We observe that the vertices of a polytope are exactly the singletons that are faces.

\begin{remark}[Computing faces and vertices of a polytope]\label{rem:facescomp}
If $f$ is a linear form in $\R^{n}$ and $\mathfrak{p} = \conv A \subseteq \R^n$, then 
\begin{equation}\label{eq:minconv}
\mathfrak{p}_f = \conv\{ \arg\min_{a\in A} f(a) \} \, . 
\end{equation}

In particular, the vertices $V$ of $\conv A $ satisfy $V \subseteq A$.
Also, when computing $ \mathfrak{p}_f$, it suffices to evaluate $f$ on $A$.
\end{remark}

\subsubsection*{Directed graphs}

All graphs, their subgraphs and their subtrees are considered to be directed multigraphs in this paper (and we often refer to them as directed graphs or simply as graphs).
In a directed multigraph $G=(V(G),E(G))$, the set of edges $E(G)$ is a multiset, allowing for loops and parallel edges. 
An edge $e \in E(G) $ is an oriented pair of vertices, $(v, u)$, often denoted by $e=v \to u$. We write $\st(e)$ for the starting vertex $v$ and $\ar(e)$ for the arrival vertex $u$.
We often consider directed graphs $G$ with labeled edges, and write $\lb( e ) $ for the label of the edge $e\in E(G)$.
In a graph with labeled edges we refer to edges by using their labels.
Given an edge $e=v \to u\in E(G)$, we denote by $C_{G}(e)$ (for ``set of \textit{continuations} of $e$'') the set of edges $e' \in E(G)$ such that $e'=u\to w$ for some $w\in V(G)$, i.e.\  $C_{G}(e) = \{e'\in E(G) | \st(e') = \ar(e) \}$.

A \textit{walk} of size $k $ on a directed graph $G$ is a sequence of $k$ edges $(e_1, \dots , e_k)\in E(G)^{k}$ such that for all $i\in [k-1]$, $\ar(e_i)=\st(e_{i+1}).$
A walk is a \textit{cycle} if $\st(e_{1}) = \ar(e_k)$.
A walk is a \textit{path} if all the edges are distinct, as well as its vertices, with a possible exception that $\st ( e_1) = \ar ( e_k ) $ may happen.
A cycle that is a path is called a \textit{simple cycle}. 
Given two walks $w=(e_1, \dots , e_k)$ and $w'=(e'_1, \dots , e'_{k'})$ such that $\ar(e_k)=\st(e'_1)$, we write $w\star w'$ for the concatenation of the two walks, i.e.\ $w \star w'=(e_1, \dots , e_k,e'_1, \dots , e'_{k'})$.
For a walk $w$, we denote by $|w|$ the number of edges in $w$.

Given a walk $w=(e_1, \dots , e_k)$ and an edge $e$, we denote by $n_e(w)$ the number of times the edge $e$ is traversed in $w$, i.e.\ $n_e(w) \coloneqq | \{i\leq k|e_i=e\}| $. 

For a vertex $v$ in a directed graph $G$, we define $\deg^i_{G} (v)$ to be the number of incoming edges to $v$, i.e.\ edges $e\in E(G)$ such that $\ar(e) = v$, and $\deg^o_{G} (v)$ to be the number of outgoing edges in $v$, i.e.\ edges $e\in E(G)$ such that $\st(e) = v$.
Whenever it is clear from the context, we drop the subscript $G$.

The \emph{incidence matrix} of a directed graph $G$ is the matrix $L(G)$ with rows indexed by $V(G)$, and columns indexed by $E(G)$, such that for any edge $e=v \to u$ with $v \neq u$, the corresponding column in $L(G)$ has $(L(G))_{v, e} = -1$, $ (L(G)))_{u, e} = 1$ and is zero everywhere else. Moreover, if $e=v \to v$ is a loop, the corresponding column in $L(G)$ has $(L(G))_{v, e} = 1$ and is zero everywhere else.

For instance, we show in \cref{fig:triangraph} a graph $G$ with its incidence matrix $L(G)$.

\begin{figure}[htbp]
	\begin{center}
		\begin{equation*}
		G=\begin{array}{ccc}
		\includegraphics[scale=0.7]{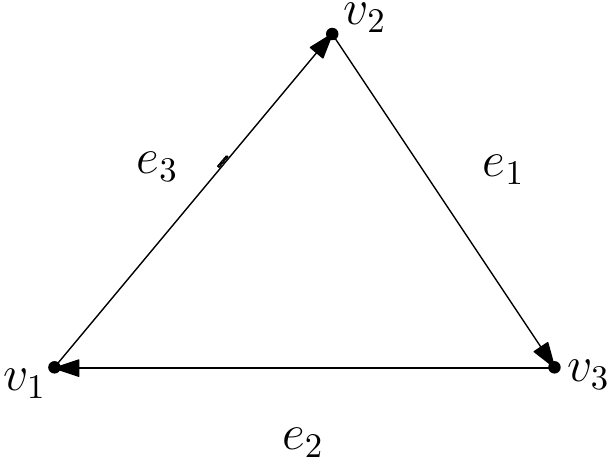}
		\end{array},
		\qquad L(G)=
		\begin{blockarray}{*{3}{c} l}
		\begin{block}{*{3}{>{$\footnotesize}c<{$}} l}
		$e_1$ & $e_2$ & $e_3$ & \\
		\end{block}
		\begin{block}{[*{3}{c}]>{$\footnotesize}l<{$}}
		0 & 1 & -1 & $v_1$ \\
		-1 & 0 & 1 & $v_2$ \\
		1 & -1 & 0 & $v_3$ \\
		\end{block}
		\end{blockarray}\,.
		\end{equation*}\\
		\caption{A graph $G$ with its incidence matrix $L(G)$. \label{fig:triangraph}}
	\end{center}
\end{figure}

A directed graph $G$ is said to be \textit{strongly connected} if for $v_1, v_2 \in V(G)$, there is a path starting in $v_1$ and arriving in $v_2$.
For instance, the graph in \cref{fig:triangraph} is strongly connected. 

\section{The cycle polytope of a graph}
\label{sec:Cycle_pol_dim}

In this section we establish general results about the cycle polytope of a graph, and recover some facts already present in the literature.
Here, all graphs are considered to be directed multigraphs that may have loops and parallel edges, unless stated otherwise. We recall the definition of cycle polytope.

\begin{definition}[Cycle polytope] \label{defn:cyc_pol}
	Let $G$ be a directed graph.
	For each non-empty cycle $\mathcal{C}$ in $G$, define $\vec{e}_{\mathcal{C}}\in \mathbb{R}^{E(G)}$ so that
	$$(\vec{e}_{\mathcal{C}})_e \coloneqq \frac{n_e(\sc)}{|\mathcal{C}|}, \quad \text{for all}\quad e\in E(G). $$
	We define $P(G) \coloneqq \conv \{\vec{e}_{\mathcal{C}} | \, \mathcal{C} \text{ is a simple cycle of } G \}$ to be the \emph{cycle polytope} of $G$.
\end{definition}

We additionally consider here the \emph{cycle cone} $\mathbb{K}(G)$ defined as follows
$$\mathbb{K}(G) \coloneqq \left\{\sum_i \alpha_i \vec{e}_{\CCC} | \, \mathcal{C} \text{ is a simple cycle of } G, \, \alpha_i \geq 0 \right\}\, . $$

We remark that this cone has already been studied in the literature.
Specifically, it is the \emph{``non-negative cone of the cycle space''} in \cite{gleiss2001circuit}, according to \cite[Proposition 4]{gleiss2001circuit}.

\subsection{Vertices of the cycle polytope}
We start by giving a full description of the vertices of this polytope.

\begin{proposition}
	\label{prop:vertices}
The set of vertices of $P(G) $ is precisely $\{\vec{e}_{\mathcal{C}} | \, \mathcal{C} \text{ a simple cycle of } G \}$.
\end{proposition}

This result was established in \cite[Proposition 4]{gleiss2001circuit} in the context of cones, that is by describing the \emph{``extreme vectors''} of $\mathbb{K}(G)$.
For sake of completeness we present here a short and self contained proof.

\begin{proof}
From Remark \ref{rem:facescomp} we only need to show that any point of the form $\vec{e}_{\sc}$ is indeed a vertex.
Consider now a simple cycle $\sc $, and recall that vertices of a polytope are characterized by being the only singletons that are faces.
Define:
$$f_{\sc}(\vec{x} ) \coloneqq - \sum_{e \in \sc} x_{e },\quad\text{for all}\quad \vec{x}=(x_e)_{e\in E(G)}\in\R^{E(G)},$$
where we identify $\sc $ with the set of edges in $\sc $.
We will show that $P(G)_{f_{\sc}} = \{\vec{e}_{\sc } \}$.
That is, that $\vec{e}_{\sc }$ is the unique minimizer of $f_{\sc}$ in $P(G)$, concluding the proof.

It is easy to check that $f_{\sc}(\vec{e}_{\sc} )= -1 $.
From \cref{eq:minconv} on page \pageref{eq:minconv}, we only need to establish that any simple cycle $\tilde{\sc}$ that satisfies $f_{\sc}(\vec{e}_{\tilde{\sc}} )\leq -1 $ is equal to $\sc $.
Take a generic simple cycle $\tilde{\sc}$ in $G$ such that $f_{\sc}(\vec{e}_{\tilde{\sc}})\leq -1$.
Then, $ \sum_{e \in \sc} (\vec{e}_{\tilde{\sc}})_{e } \geq 1 $.
Since $\vec{e}_{\tilde{\sc}}$ satisfies the equation $\sum_{e\in E(G)} {(\vec{e}_{\tilde{\sc}})}_{e } = 1$ and has non-negative coordinates, we must have that $(\vec{e}_{\tilde{\sc}})_e = 0$ for all $e\not\in \sc $. 
Thus $\tilde{\sc} \subseteq \sc $ as sets of edges.
However, because both $\tilde{\sc}, \sc$ are simple cycles, we conclude that $\sc = \tilde{\sc}$, as desired.
\end{proof}

\subsection{Dimension of the cycle polytope}
We present the proof of the following result here:

\begin{proposition}[Dimension of the cycle polytope]\label{thm:dimension}
	If $G$ is a strongly connected graph, then the cycle polytope of $G$ has dimension $|E(G)| - | V(G)| $.
\end{proposition}

\begin{proof}
In \cite[Proposition 6]{gleiss2001circuit}, it was established that there exists a \emph{``cycle basis''} for $\mathbb{K}(G)$, that is a set of linearly independent vectors $\{\vec{f}_i\}$ in $\mathbb{K}(G)$ that span $\mathbb{K}(G)$.
By rescaling these vectors appropriately, we consider such a basis $\{\vec{f}_i\}_{i=1}^L$ in the hyperplane $\{\sum x_e =1\}$.
Thus we have that $\{\vec{f}_i\}_{i=1}^L\subseteq P(G)$.

From \cite[Eq. (3)]{gleiss2001circuit}, $\mathbb{K}(G)$ spans a space of dimension $L = |E(G)| - |V(G)| + 1$.
Because  the vectors $\{\vec{f}_i\}_{i=1}^L$ are linearly independent, they are affinely independent, and so $$\dim P(G) \geq  (|E(G)| - |V(G)| + 1) -1 =|E(G)| - |V(G)|.$$

Conversely, any affinely independent set $S$ of vectors of $P(G)$ forms a linearly independent set in $\mathbb{K}(G)$, so $|S| \leq |E(G)| - |V(G)| + 1$.
This concludes that $\dim P(G) = |E(G)| - |V(G)| $.
\end{proof}

We now generalize Proposition \ref{thm:dimension} to any graph. We start with the following technical result.

\begin{lemma}\label{prop:abc}
Let $\mathfrak{a}_1\subset \R^A, \mathfrak{a}_2\subset \R^B$ be polytopes such that $ \Aff(\mathfrak{a}_1), \Aff(\mathfrak{a}_2)$ do not contain the zero vector, and let $d(\mathfrak{a}_1), d(\mathfrak{a}_2)$ be their respective dimensions.

Then the dimension of the polytope $\mathfrak{c} = \conv( \mathfrak{a}_1\times \{\vec{0}\} , \{\vec{0}\} \times \mathfrak{a}_2)\subset \R^{A\sqcup B}$ is 
$$d(\mathfrak{c})= d(\mathfrak{a}_1) + d(\mathfrak{a}_2) + 1.$$
\end{lemma}

This result is tipically regarded as a well known fact in convex geometry.
We present a proof for sake of completeness.

\begin{proof}
In this proof, for sake of simplicity, we will identify $\mathfrak{a}_1 \subset \mathbb{R}^A$ and $\mathfrak{a}_2 \subset \mathbb{R}^B$ with $\mathfrak{a}_1\times \{\vec{0}\}$ and $ \{\vec{0}\}\times\mathfrak{a}_2$, respectively.
In particular, we will refer to points $\vec{x}\in \mathfrak{a}_i $ for $i=1, 2$ as their suitable extensions $(\vec{x}, \vec{0})$ or $( \vec{0}, \vec{x})$, respectively, in $\R^{A\sqcup B} $ without further notice.

We start with a lower bound for $d(\mathfrak{c})$.
Consider affinely independent sets 
\begin{equation*}
\mathcal{V}_1 = \{\vec{v}_1^{(1)}, \dots, \vec{v}_{d(\mathfrak{a}_1)+1}^{(1)} \} \quad \text{and} \quad \mathcal{V}_2 = \{ \vec{v}_1^{(2)}, \dots, \vec{v}_{d(\mathfrak{a}_2)+1}^{(2)} \} 
\end{equation*} 
in $\mathfrak{a}_1$ and $\mathfrak{a}_2$, respectively. 

Given an affine space $\mathcal A$ not containing the origin, and a set of vectors $S\subset \mathcal A$, it is easily seen that $S$ is linearly independent if and only if it is affinely indepentent.
Therefore, each of the sets $\mathcal{V}_1, \mathcal{V}_2 $ is linarly independent, because $\mathcal{V}_i \subseteq \Aff (\mathfrak{a}_i )$ and $\Aff (\mathfrak{a}_i )$ does not contain $\vec{0}$ for $i=1, 2$.
Since $\spn (\mathfrak{a}_1 ) \perp \spn (\mathfrak{a}_2 )$, the union $\mathcal{V}_1 \cup \mathcal{V}_2$ is also a linearly independent set, and so also affinely independent.
Observe that we found an affinely independent set $\mathcal V_1 \cup \mathcal V_2$ with $d(\mathfrak{a}_1) + d(\mathfrak{a}_2) + 2$ many vectors in $\mathfrak{c}$.
This gives us the lower bound $d( \mathfrak{c}) \geq d(\mathfrak{a}_1) + d(\mathfrak{a}_2) + 1$.

For an upper bound, note that $\Aff (\mathfrak{c}) \subseteq \spn (\mathfrak{c})$, and that 
$$\dim (\spn (\mathfrak{c})) \leq \dim (\spn (\mathfrak{a}_1)) + \dim (\spn (\mathfrak{a}_2)) = d(\mathfrak{a}_1) + d(\mathfrak{a}_2) + 2\, . $$
We now prove that $0 \not\in \Aff (\mathfrak{c}) $ by contradiction.
Assume otherwise that we have $\sum_i \alpha_i \vec{a}_i + \sum_j \beta_j \vec{b}_j = 0$, where $\vec{a}_i \in \mathfrak{a}_1$, $\vec{b}_j \in \mathfrak{a}_2$ and $\sum_i\alpha_i + \sum_j \beta_j=1$.
But we have that $\sum_i \alpha_i \vec{a}_i  \in \R^A$, and $ - \sum_j \beta_j \vec{b}_j \in \R^B$, so  $\sum_i  \alpha_i \vec{a}_i = - \sum_j \beta_j \vec{b}_j \in \R^A \cap \R^B = \{ \vec{0} \}$.
Because $\sum_i\alpha_i + \sum_j \beta_j=1$, without loss of generality we can assume that $\sum_i\alpha_i\neq 0$.
Then we have $\frac{\sum_i \alpha_i \vec{a}_i}{\sum_i \alpha_i } = 0 \in \Aff (\mathfrak{a}_1)$, a contradiction.

Since $0\in \spn (\mathfrak{c})$, we conclude that $\Aff (\mathfrak{c}) \neq \spn \mathfrak{c}$. It follows that 
$$ \dim \left(\Aff (\mathfrak{c}) \right) \leq d(\mathfrak{a}_1) + d(\mathfrak{a}_2) + 1\, . $$
This concludes the proof.
\end{proof}

With the help of Lemma \ref{prop:abc}, we can generalize Proposition \ref{thm:dimension} to the cycle polytope of any graph:
We say that a graph $G=(V, E)$ is \textit{full} if any edge $e \in E$ is part of a cycle of $G$.
It is easy to see that if $G$ is not full, then $P(G) = P(H)$, where $H\subseteq G$ is the largest full subgraph of $G$.
Equivalently, $H$ is obtained from $G$ by removing all edges from $G$ that are not part of a cycle.

If $G$ is a full graph, there are no \textit{bridges}, that is an edge $e$ connecting two distinct strongly connected components.
Hence, we can decompose $G$ as the disjoint union of strongly connected components and a set of isolated vertices $V'$: $G = H_1 \sqcup \dots \sqcup H_k \sqcup V'$.
It can be seen that $P(G) = \conv \{P(H_i)| i=1, \dots k\}$, where we identify $P(H_i)$ with its canonical image in $\R^{E(G)}$.
Noting that $P(H_i)\subseteq\Aff (P(H_i))\subseteq\R^{E(H_i)} $ and that $\Aff (P(H_i))$ does not contain the origin for any $i=1, \dots , k$, from Lemma \ref{prop:abc} we have that
$$\dim P(G) = k - 1 + \sum_i \dim P(H_i ) = k - 1 + |E| - |V\setminus V'| =  |E| - |V| + |V'| + k -1  \, . $$

\begin{proposition}\label{thm:dimfullgr}
If $G$ is a directed multigraph and $H\subseteq G$ its largest full subgraph, then the dimension of the polytope $P(G)$ is
$$\dim P(G) = |E(H)| - |V(G)| + |\{ \text{ connected components of } H \} |- 1. $$
\end{proposition}

\subsection{Faces of the cycle polytope}
We now focus on the faces of a cycle polytope $P(G)$. We prove two results: in Proposition \ref{thm:faces} we describe the equations that define $P(G)$, then in Theorem \ref{cor:facestruct} we find a bijection between faces of $P(G)$ and the subgraphs of $G$ that are full.
\begin{proposition}\label{thm:faces}
Let $G$ be a directed graph.
The polytope $P(G)$ is given by
$$ P(G) = \left\{ \vec{x}\in {[0,1]}^{E(G)} \Bigg| \sum_{e\in E(G)} x_e = 1\, , \, \, \sum_{\st (e) = v} x_e = \sum_{\ar (e) = v} x_e \, , \forall v \in V(G)\right\} \, .$$
\end{proposition}

\begin{proof}
Consider $L(G)$ the incidence matrix of the graph $G$ and note that the equation $L(G) \vec{x} = \vec{0}$ is equivalent to $\sum_{\st (e) = v} x_e = \sum_{\ar (e) = v} x_e \, , \forall v \in V(G)$.
From \cite[Proposition 1]{gleiss2001circuit}, we have that $\mathbb{K}(G)= \{\vec{x} \geq 0 \} \cap \ker L(G) $.
It is immediate to see that $P(G) = \mathbb{K}(G) \cap \{\sum_i x_i =  1\}$, and so the result follows.
\end{proof}

Recall that a subgraph $H=(V, E')$ of a graph $G$ is called a \textit{full subgraph} if any edge $e \in E'$ is part of a cycle of $H$.

\begin{theorem}
\label{cor:facestruct}
The face poset of $P(G)$ is isomorphic to the poset of non-empty full subgraphs of $G$ according to the following identification:
$$H \mapsto P(G)_H \coloneqq \{\vec{x}\in P(G) | x_e = 0 \text{ for } e\not\in E(H) \} \, .$$
Further, if we identify $P(H)$ with its image under the inclusion $\R^{E(H)} \hookrightarrow \R^{E(G)}$, we have that $P(H) =  P(G)_H$.

In particular, $\dim (P(G)_H ) = |E(H)| - |V| + |\{ \text{ connected components of } H \} |- 1$.
\end{theorem}

\begin{figure}[h]
\includegraphics[scale=0.55]{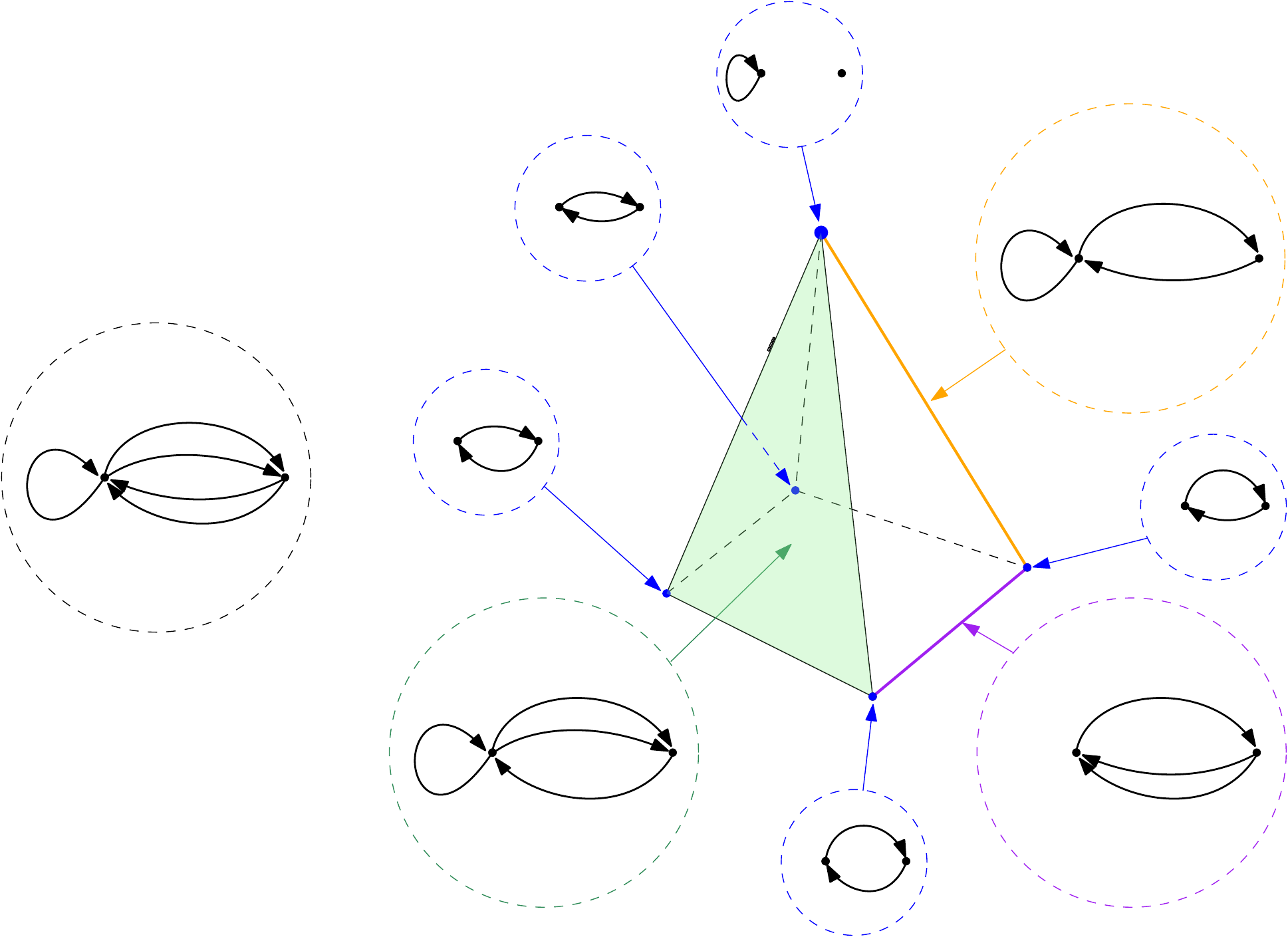}
\caption{The face structure of the cycle polytope of a graph. On the left-hand side of the picture (inside the dashed black ball) we have a graph $G$ with two vertices and five edges. On the right-hand side, we draw the associated cycle polytope $P(G)$ that is a pyramid with squared base. The blue dashed balls correspond to the simple cycles corresponding to the five vertices of the polytope. We also underline the relation between two edges of the polytope (in purple and orange respectively) and a face (in green) and the corresponding full subgraphs. Note that, for example, the graph corresponding to the green face is just the union of the three graphs corresponding to the vertices of that face.  \label{fig:facestruct}}
\end{figure}

\begin{proof}
From Proposition \ref{thm:faces}, a face of $P(G)$ is given by setting some of the inequalities of $\vec{x} \geq 0$ as equalities.
So, a face is of the form $P(G)_H$ for some subgraph $H=(V(G), E(H))$, where $E(G)\setminus E(H)$ corresponds to the inequalities of $\vec{x} \geq 0 $ that become equalities.
It is immediate to observe that the identification $\R^{E(H)} \hookrightarrow \R^{E(G)}$ gives us that $P(H) =  P(G)_H$.

We show that it suffices to take $H$ a full subgraph: consider an edge $e_0 = v \to w$ in $H$ that is not contained in any cycle in $H$.
Then $(\vec{e}_{\sc})_{e_0} = 0 $ for any simple cycle $\sc$  in $H$, and so  $x_{e_0} = 0 $ for any point $\vec{x} \in P(H) = P(G)_H$.
It follows that $ P(G)_{H} =  P(G)_{H\setminus e_0}$.

Conversely, we can see that if $H_1\neq H_2$ are two full subgraphs of $G$, then $P(G)_{H_1}~\neq~P(G)_{H_2}$, that is, $H_1, H_2$ correspond to two different faces of $P(G)$.
Indeed, without loss of generality we can assume that there is an edge $e \in E(H_1)\setminus E(H_2)$.
This edge is, by hypothesis, contained in a simple cycle $\sc$, so $\vec{e}_{\sc}\in P(G)_{H_1}\setminus P(G)_{H_2}$, so $P(G)_{H_1}~\neq~P(G)_{H_2}$.

It is clear that if $H_1\subseteq H_2$ then $P(G)_{H_1}\subseteq P(G)_{H_2}$, so we have that the identification $H \mapsto P(G)_H$ preserves the poset structure.
Finally, we obtained the dimension of $P(G)_H = P(H)$ in Proposition \ref{thm:dimfullgr}.
\end{proof}

\begin{example}[Face structure of a specific cycle polytope]
Consider the graph $G$ given on the left-hand side of \cref{fig:facestruct}, that has two vertices and five edges.
It follows that the corresponding cycle polytope has dimension three, and its face structure is partially depicted in the right-hand side \cref{fig:facestruct}.

In fact, from Theorem \ref{cor:facestruct}, to each face of the polytope we can associate a full subgraph of $G$.
Some of these correspondences are highlighted in \cref{fig:facestruct} and described in its caption.
\end{example}

Given a simple cycle $\sc $ in a graph $G$, a path $P$ is a \textit{chord} of $\sc $ if it is edge-disjoint from $\sc $ and it starts and arrives at vertices of $\sc $.
In particular, given two simple cycles sharing a vertex, any one of them forms a chord of the other.

\begin{remark}[The skeleton of the polytope $P(G)$]
We want to characterize the pairs of vertices of $P(G)$ that are connected by an edge.
The structure behind this is usually called the \emph{skeleton of the polytope}.
Suppose that we are given two vertices of the polytope $P(G)$, $\vec{e}_{\sc_1}, \vec{e}_{\sc_2}$ corresponding to the simple cycles $\sc_1, \sc_2$ of the graph $G$, according to Proposition \ref{prop:vertices}.

With the description of the faces in Theorem \ref{cor:facestruct}, we have that a face $P(G)_H$ is an edge when it has dimension one, that is 
$$|E(H)| - |V(G)| + |\{ \text{ connected components of } H \} |- 1 = 1 \, . $$
This happens if and only if the undirected version of $H$ is a forest with two edges added.

Because $H$ is full, each connected component must contain a cycle, so $H$ has either one or two connected components.
Hence, it results either from the union of two vertex-disjoint simple cycles, or from the union of a simple cycle and one of its chords.
Equivalently, $\vec{e}_{\sc_1}, \vec{e}_{\sc_2}$ are connected with an edge when $\sc_1 \setminus \sc_2$ forms a unique chord of $\sc_2$, or when $\sc_1, \sc_2$ are vertex-disjoint.

For instance, in \cref{fig:facestruct}, there are two pairs of vertices of $P(G)$ that are not connected, and each pair corresponds to two cycles $\sc_1, \sc_2$ such that $\sc_1 \setminus \sc_2$ forms two chords of $\sc_2$.
\end{remark}

\begin{remark}[Computing the volume of $P(G)$]
The problem of finding the volume of a polytope is a classical one in convex geometry.
We propose an algorithmic approach that uses the face description of $P(G)$ in Theorem \ref{cor:facestruct} and the following facts:

\begin{itemize}

\item Let $A$ be a polytope and $v$ a point in space. 
If $v \not\in \Aff(A)$, then
$$ \vol(\cnv(A \cup \{ v \} ) ) = \vol(A) \dist(v , \Aff(A))\frac{1}{\dim A + 1}\, . $$

\item If $v  $ is a vertex of the polytope $\mathfrak{p}$ of dimension $d$, then we have the following decomposition of the polytope $\mathfrak{p}$: 
$$\mathfrak{p} = \bigcup_{v\not\in\mathfrak{q} \subsetneq \mathfrak{p}  } \conv( \mathfrak{q}\cup \{ v \} )\, , $$ 
where the union runs over all high-dimensional faces $\mathfrak{q}$ that do not contain the vertex $v$.
This decomposition is such that the intersection of each pair of blocks has volume zero, and each block has a non-zero $d-1$ dimensional volume.
\end{itemize}

If $\sc $ is a simple cycle of $G$, the following decomposition holds:
$$P(G) = \bigcup_{\sc \not\subseteq H \subsetneq G  } \conv(\vec{e}_{\sc} , P(G)_H )\, , $$ 
where the union runs over all maximal full proper subgraphs of $G$ that do not contain $\sc$.

Hence, we obtain the volume of $P(G)$ as follows:
$$\vol ( P(G) ) = \sum_{\sc \not\subseteq H \subsetneq G  } \conv(\vec{e}_{\sc} , P(G)_{G\setminus e} ) =\sum_{\sc \not\subseteq H \subsetneq G  } \frac{\vol(P(H)) \dist(\vec{e}_{\sc} , \Aff (P(G)_H))}{\dim P(G) + 1},$$
 where the sum runs over all maximal full proper subgraphs of $G$ that do not contain $\sc $.
This gives us a recursive way of computing the volume $ \vol(P(G)) $ by computing the volume of cycle polytopes of smaller graphs.
We have unfortunately not been able to find a general formula for $\vol P(G)$, and leave this as an open problem.
\end{remark}

\section{The feasible region \texorpdfstring{$P_k$}{} is a cycle polytope}
\label{sec:Pn_is_poly}

Recall that we defined
$$ 
P_k \coloneqq \left\{\vec{v}\in [0,1]^{\SS_k} \big| \exists (\sigma^m)_{m\in\NN} \in \SS^{\NN} \text{ s.t. }|\sigma^m| \to
\infty \text{ and }  \pcoc(\pi, \sigma^m ) \to \vec{v}_{\pi}, \forall \pi\in\SS_k  \right\}  .
$$

The goal of this section is to prove that $P_k$ is the cycle polytope of the overlap graph $\ValGraph[k]$ (see Theorem \ref{thm:vertices}). 
We first prove that $P_k$ is closed and convex (see Proposition \ref{prop:P_n_is_covex}), then we use a correspondence between permutations and paths in $\ValGraph[k]$ (see Definition \ref{defn:overlap_graph}) to prove the desired result.

\subsection{The feasible region \texorpdfstring{$P_k$}{} is convex}
We start with a preliminary result.

\begin{lemma}
	\label{lem:P_n_is_closed}
	The feasible region $P_k$ is closed.
\end{lemma}

This is a classical consequence of the fact that $P_k$ is a set of limit points. For completeness, we include a simple proof of the statement. Recall that we defined $\pcoc_k ( \sigma )\coloneqq \left(\pcoc(\pi,\sigma)\right)_{\pi\in\SS_k}$.

\begin{proof}
	It suffices to show that, for any sequence $(\vec{v}_s)_{s\in\NN}$ in $P_k$ such that $\vec{v}_s\to\vec{v}$ for some $\vec{v}\in [0,1]^{\SS_k}$, we have that $\vec{v}\in P_k$. 
	For all $s\in\NN$, consider a sequence of permutations $(\sigma^m_s)_{m\in\NN}$ such that $|\sigma^m_s|\stackrel{m\to\infty}{\longrightarrow}\infty$ and $\pcoc_k( \sigma^m_s)\stackrel{m\to\infty}{\longrightarrow}\vec{v}_s$, and some index $m( s )$ of the sequence $(\sigma^m_s)_{m\in\NN}$ such that for all $m\geq m(s),$
	$$|\sigma^{m}_s|\geq s\quad\text{and}\quad ||\pcoc_k( \sigma^{m}_s)-\vec{v}_s||\leq\tfrac{1}{s}\, .$$
	
	Without loss of generality, assume that $m(s)$ is increasing. For every $\ell\in\NN$, define $\sigma^{\ell}\coloneqq\sigma^{m(\ell)}_\ell$.
	It is easy to show that 
	$$|\sigma^\ell|\stackrel{\ell\to\infty}{\longrightarrow}\infty \quad\text{and}\quad\pcoc_k(\sigma^\ell)\stackrel{\ell\to\infty}{\longrightarrow}\vec{v}\, ,$$ 
	where we use the fact that $\vec{v}_s\to\vec{v}$. Therefore $\vec{v}\in P_k$.
\end{proof}

We can now prove the first important result of this section.

\begin{proposition}
	\label{prop:P_n_is_covex}
	The feasible region $P_k$ is convex.
\end{proposition}

\begin{proof}
	Since $P_k$ is closed (by Lemma \ref{lem:P_n_is_closed}) it is enough to consider rational convex combinations of points in $P_k$, i.e.\ it is enough to establish that for all $\vec{v}_1,\vec{v}_2\in P_k$ and all $ s,t\in\NN$, we have that
	\begin{equation*}
		\frac{s}{s+t}\vec{v}_1+\frac{t}{s+t}\vec{v}_2\in P_k.
	\end{equation*}
	Fix $\vec{v}_1,\vec{v}_2\in P_k$ and $s,t\in\NN$. Since $\vec{v}_1,\vec{v}_2\in P_k$, there exist two sequences $(\sigma^m_1)_{m\in\NN}$, $(\sigma^m_2)_{m\in\NN}$ such that $|\sigma^m_i|\stackrel{m\to\infty}{\longrightarrow}\infty$ and $\pcoc_k( \sigma^m_i)\stackrel{m\to\infty}{\longrightarrow}\vec{v}_i$, for $i=1,2$.
	
	Define $t_m\coloneqq t\cdot |\sigma^m_1|$ and $s_m\coloneqq s\cdot |\sigma^m_2|$. 
	
\begin{figure}[htbp]
		\begin{center}
			\includegraphics[scale=.65]{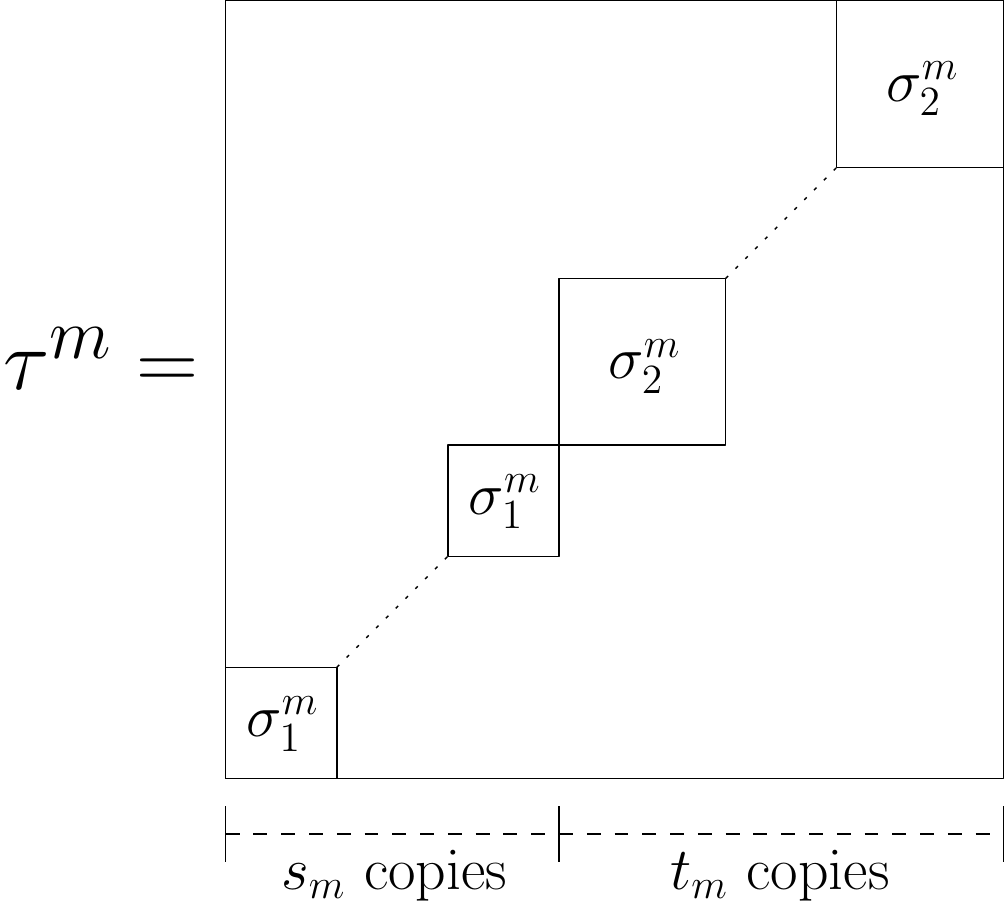}\\
			\caption{Schema for the definition of the permutation $\tau^m$. \label{schema_direct_sum}}
		\end{center}
	\end{figure}	
	We set $\tau^m\coloneqq \left(\oplus_{s_m}\,\sigma^m_1\right)\oplus\left(\oplus_{t_m}\,\sigma^m_2\right)$. For a graphical interpretation of this construction we refer to \cref{schema_direct_sum}.
	We note that for every $\pi\in\SS_k$, we have
	\begin{equation*}
		\coc(\pi,\tau^m)=s_m\cdot\coc(\pi,\sigma^m_1)+t_m\cdot\coc(\pi,\sigma^m_2)+Er,
	\end{equation*}
	where $Er\leq(s_m+t_m-1)\cdot |\pi|$. This error term comes from the number of intervals of size $|\pi|$ that intersect the boundary of some copies of $\sigma^m_1$ or $\sigma^m_2$. Hence
	\begin{equation*}
	\begin{split}
	\pcoc(\pi,\tau^m)&=\frac{s_m\cdot|\sigma^m_1|\cdot\pcoc(\pi,\sigma^m_1)+t_m\cdot|\sigma^m_2|\cdot\pcoc(\pi,\sigma^m_2)+Er}{s_m\cdot|\sigma^m_1|+t_m\cdot |\sigma^m_2|}\\
	&=\frac{s}{s+t}\pcoc(\pi,\sigma^m_1)+\frac{t}{s+t}\pcoc(\pi,\sigma^m_2)+O\left(|\pi|\left(\tfrac{1}{|\sigma^m_1|}+\tfrac{1}{|\sigma^m_2|}\right)\right).
	\end{split}
	\end{equation*}
	As $m$ tends to infinity, we have $$\pcoc_k(\tau^m)\to\frac{s}{s+t}\vec{v}_1+\frac{t}{s+t}\vec{v}_2,$$ 
	since $|\sigma^m_i|\stackrel{m\to\infty}{\longrightarrow}\infty$ and $\pcoc_k( \sigma^m_i)\stackrel{m\to\infty}{\longrightarrow}\vec{v}_i$, for $i=1,2$. Noting also that
	$$|\tau^m|\to\infty,$$
	we can conclude that $\tfrac{s}{s+t}\vec{v}_1+\tfrac{t}{s+t}\vec{v}_2\in P_k$. This ends the proof.
\end{proof}

\subsection{The feasible region \texorpdfstring{$P_k$}{} as the limit of random permutations}

Using similar ideas to the ones used in the proof above, we can establish the equality between sets
\begin{align*}
P_k \coloneqq &\left\{\vec{v}\in [0,1]^{\SS_k} \big| \exists (\sigma^m)_{m\in\NN} \in \SS^{\NN} \text{ s.t. }|\sigma^m| \to
\infty \text{ and }  \pcoc(\pi, \sigma^m ) \to \vec{v}_{\pi}, \forall \pi\in\SS_k \right\}\\
=&\left\{(\Gamma_{\pi}(\sigma^{\infty}))_{\pi\in\SS_k} \big| \sigma^{\infty}\text{ is a random infinite rooted \emph{shift-invariant} permutation}  \right\} \, . 
\end{align*}
We first recall the following.

\begin{definition}\label{def_equality_intro}
	For a total order $(\Z,\preccurlyeq)$, its \emph{shift} $(\Z,\preccurlyeq')$ is defined by $i+1\preccurlyeq' j+1$ if and only if $i\preccurlyeq j.$ A random infinite rooted permutation, or equivalently a random total order on $\Z$, is said to be shift-invariant if it has the same distribution as its shift.
\end{definition}

We refer to \cite[Section 2.6]{borga2018local} for a full discussion of shift-invariant random permutations.

\begin{proposition}\label{cor_equality_intro} The following equality holds 
	\begin{equation*}
	P_k=\left\{(\Gamma_{\pi}(\sigma^{\infty}))_{\pi\in\SS_k} \big| \sigma^{\infty}\text{ is a random infinite rooted shift-invariant permutation}  \right\} \, . 
	\end{equation*}
\end{proposition}

\begin{proof}
	In \cite[Proposition 2.44 and Theorem 2.45]{borga2018local} it was proved that a random infinite rooted permutation $\sigma^\infty$ is shift-invariant  if and only if it is the annealed Benjamini--Schramm limit of a sequence of random permutations\footnote{The annealed Benjamini--Schramm convergence is an extension of the Benjamini--Schramm convergence (BS-limit) to sequences of \emph{random} permutations. For more details see \cite[Section 2.5.1]{borga2018local}.}.  
	Furthermore, we can choose this sequence of random permutations $\sigma^n$ in such a way that $|\sigma^n|=n $ a.s., for all $n\in\NN$.
	
	This result and Theorem \ref{thm:carac_cocc} immediately imply that 
	\begin{equation*}
	P_k\subseteq\left\{(\Gamma_{\pi}(\sigma^{\infty}))_{\pi\in\SS_k} \big| \sigma^{\infty}\text{ is a random infinite rooted shift-invariant permutation}  \right\} \, . 
	\end{equation*}
	To show the other inclusion, it is enough to show that for every random infinite rooted shift-invariant permutation $\sigma^{\infty}$, there exists a sequence of \emph{deterministic} permutations that Benjamini--Schramm converges to $\sigma^{\infty}$.
	
	By the above mentioned result of \cite{borga2018local}, there exists a sequence $(\sigma^n)_{n\in\NN}$ of \emph{random} permutations such that $|\sigma^n|=n$ a.s., for all $n\in\NN$, and $(\sigma^n)_{n\in\NN}$ converges in the annealed Benjamini--Schramm sense to $\sigma^{\infty}$. Using \cite[Theorem 2.24]{borga2018local} we know that, for every $\pi\in\SS$,
	\begin{equation}\label{eq:first_conv}
	\mathbb{E}[\pcoc(\pi,\sigma^n)]\to\Gamma_{\pi}(\sigma^\infty)\, .
	\end{equation}
	Let, for all $n\in\NN$ and $\rho\in\SS_n$,
	$$p^n_\rho\coloneqq\mathbb P(\sigma^n=\rho).$$
	For every $n\in\NN$, we can find $n!$ integers $\{q^n_\rho\}_{\rho\in\SS_n}$ such that for every $\rho\in\SS_n$,
	\begin{equation}
	\label{eq:bound}
	\left|\frac{q^n_\rho}{\sum_{\theta\in\SS_n}q^n_\theta}-p^n_\rho\right|\leq\frac{1}{n^n}.
	\end{equation}
	Let us now consider the deterministic sequence of permutations of size $n \sum_{\theta \in \SS_n} q_{\theta}^n$ defined as $$\nu^n\coloneqq\bigoplus_{\rho\in\SS_n}(\oplus_{q^n_\rho}\rho),$$
	where we fixed any order on $\SS_n$.
	Using the same error estimates as in the proof of Proposition \ref{prop:P_n_is_covex}, it follows that
	$$\pcoc(\pi,\nu^n)=\frac{\sum_{\rho\in\SS_n}q^n_\rho\cdot\coc(\pi,\rho)+Er}{n\cdot\sum_{\theta\in\SS_n}q^n_\theta},\quad\text{for all}\quad \pi\in\SS,$$
	with $Er\leq( -1  + \sum_{\theta\in\SS_n}q^n_\theta )\cdot|\pi|$. Therefore	
	\begin{multline*}
	 \left| \pcoc(\pi,\nu^n)-\mathbb{E}[\pcoc(\pi,\sigma^n)] \right|\\
	 \leq  \left| \sum_{\rho\in\SS_n} \frac{q^n_\rho}{\sum_{\theta\in\SS_n} q^n_\theta} \cdot \pcoc(\pi, \rho) - \sum_{\rho\in\SS_n} p^n_\rho \cdot \pcoc(\pi,\rho) \right| +\left|\frac{Er}{n\cdot\sum_{\theta\in\SS_n}q^n_\theta} \right| \\
	\leq\frac{1}{n^n}\cdot\sum_{\rho\in\SS_n}\pcoc(\pi,\rho)+\frac{|\pi|}{n},
	\end{multline*}
	where in the second inequality we used the bound in \cref{eq:bound} and the bound for $Er$. Since the size of $\pi$ is fixed and the term $\sum_{\rho\in\SS_n}\pcoc(\pi,\rho)$ is bounded by $n!$, we can conclude that $\left|\pcoc(\pi,\nu^n)-\mathbb{E}[\pcoc(\pi,\sigma^n)]\right|\to0$. Combining this with \cref{eq:first_conv} we get
	$$\pcoc(\pi,\nu^n)\to \Gamma_{\pi}(\sigma^\infty), \quad\text{for all}\quad \pi\in\SS.$$
	Therefore, using Theorem \ref{thm:carac_cocc} we can finally deduce that the deterministic sequence $\{\nu^n\}_{n\in\NN}$ converges to $\sigma^\infty$ in the Benjamini--Schramm topology, concluding the proof.
\end{proof}

\subsection{The overlap graph}
We now want to study the way in which consecutive patterns of permutations can overlap.
 
We start by introducing some more notation. For a permutation $\pi \in\SS_k$, with $k\in\NN_{\geq2}$, let $\be ( \pi )\in \SS_{k-1}$ (resp.\ $\en(\pi) \in \SS_{k-1}$) be the patterns generated by its first $k-1$ indices (resp.\ last $k-1$ indices). More precisely,
\[\be ( \pi )\coloneqq\pat_{[1,k-1]}(\pi)\quad \text{and}\quad \en( \pi )\coloneqq\pat_{[2,k]}(\pi).\]

The following definition, introduced in \cite{MR1197444}, is key in the description of the feasible region $P_k$.

\begin{definition}[Overlap graph]\label{defn:overlap_graph}
	Let $k\in\NN_{\geq2}$.
	We define the \emph{overlap graph} $\ValGraph[k]$ of size $k$  as a directed multigraph with labeled edges, where the vertices are elements of $\SS_{k-1}$ and for all $\pi\in\SS_{k}$ we add the edge $ \be(\pi ) \to \en(\pi ) $ labeled by $\pi$. 
\end{definition}

This gives us a directed graph with $k!$ edges, and $(k-1)!$ vertices.
Informally, the continuations of an edge $\tau$ in the overlap graph $\ValGraph[k]$ record the consecutive patterns of size $k$ that can appear after the consecutive pattern $\tau$. 
More precisely, for a permutation $\sigma\in\SS_{\geq k+1}$ and an interval $I\subseteq[|\sigma|-1]$ of size $k$, let $\tau \coloneqq \pat_{I}(\sigma)\in\SS_{k} $, then we have that 
\begin{equation}\label{eq:contin}
\pat_{I+1}(\sigma)\in C_{\ValGraph[k]}(\tau),
\end{equation}
where $I+1$ denotes the interval obtained from $I$ shifting all the indices by $+1$, and we recall that $C_{\ValGraph[k]}(\tau)$ is the set of continuations of $\tau$.

\begin{example}
	\label{exemp:overlap_graph}
We recall that the overlap graph $\ValGraph[4]$ was displayed in \cref{Overlap_graph_exemp} on page~\pageref{Overlap_graph_exemp}. The six vertices (in red) correspond to the six permutations of size three and the twenty-four oriented edges correspond to the twenty-four permutations of size four.
\end{example}
	
Given a permutation $\sigma\in\SS_m$, for some $m\geq k$, we can associate to it a walk $W_k(\sigma)=(e_1,\dots,e_{m-k+1})$ in $\ValGraph[k]$ of size $m-k+1$ defined by 
\begin{equation}
\label{eq:def_walf}
	\lb(e_i)\coloneqq\pat_{[i,i+k-1]}(\sigma),\quad\text{for all}\quad i\in [m-k+1].
\end{equation}

Note that \cref{eq:contin} justifies that this sequence of edges is indeed a walk in the overlap graph.
\begin{example}
	\label{exemp:overlap_graph2}
	Take the graph $\ValGraph[4]$ from \cref{Overlap_graph_exemp} on page~\pageref{Overlap_graph_exemp}, and consider the permutation $\sigma=628451793\in\SS_9.$ The corresponding walk $W_4(\sigma)$ in $\ValGraph[4]$ is
	$$(3142,1423,4231,2314,2134,1342)$$ 
	and it is highlighted in green in \cref{Overlap_graph_exemp}.	
\end{example}

Note that the map $W_k$ is not injective (see for instance Example \ref{exemp:overlap_graph3} below) but the following holds. 

\begin{lemma}\label{lemma:pathperm}
	Fix $k\in\NN_{\geq2}$ and $m\geq k$. The map $W_k$, from the set $\SS_{m}$ of permutations of size $m$ to the set of walks in $\ValGraph[k]$ of size $m-k+1$, is surjective.	
\end{lemma}

\begin{proof}
	We exhibit a greedy procedure that, given a walk $w=(e_1,\dots,e_s)$ in $\ValGraph[k]$, constructs a permutation $\sigma$ of size $s+k-1$ such that $W_k(\sigma)=w$.
	Specifically, we construct a sequence of $s$ permutations $(\sigma_i)_{i\leq s}$, with $|\sigma_i|=i+k-1$, in such a way that $\sigma$ is equal to $\sigma_s$. For this proof, it is useful to consider permutations as diagrams.
	
	The first permutation is defined as $\sigma_1 = \lb (e_1 )$.
	To construct $\sigma_{i+1}$ we add to the diagram of $\sigma_i$ a final additional point on the right of the diagram between two rows, in such a way that the last $k$ points induce the consecutive pattern $\lb ( e_{i+1} )$ (the choice for this final additional point may not be unique, but exists).
	Setting $\sigma\coloneqq \sigma_s$ we have by construction that $W_k(\sigma)=w.$
\end{proof}

We illustrate the construction above in a concrete example.

\begin{example}
	\label{exemp:overlap_graph3}
	Consider the walk $w=(3142,1423,4231,2314,2134,1342)$ obtained in Example \ref{exemp:overlap_graph2} and construct, as explained in the previous proof, a permutation $\sigma$ such that $W_k(\sigma)=w$. We set 
	$\sigma_1=3142$.  
	Then, since $e_2=1423$, we add a point between the second and the third row of $\sigma_1$ (see \cref{fig:diag} for the diagrams of the considered permutations), obtaining
	$\sigma_2=41523.$ Note that the pattern induced by the last $4$ points of $\sigma_2$ is exactly $e_2=1423.$ We highlight that we could also add the point between the third and the fourth row of $\sigma_1$ obtaining the same induced pattern. However, in this example, we always chose to add the points in the bottommost possible place. We iterate this procedure constructing
	$\sigma_3=516342,$
	$\sigma_4=6173425,$
	$\sigma_5=71834256,$
	$\sigma_6=819452673.$
	Setting $\sigma\coloneqq \sigma_6=819452673$ we obtain that $W_4(\sigma)=w.$ 
	Note that this is not the same permutation considered in Example \ref{exemp:overlap_graph2}, indeed the map $W_k$ is not injective.
	\begin{figure}[htbp]
		\begin{center}
			\begin{equation*}
			\begin{split}
			\sigma_1=\begin{array}{lcr}
			\begin{tikzpicture}
			\begin{scope}[scale=.3]
			\permutation{3,1,4,2}
			\end{scope}
			\end{tikzpicture}
			\end{array},\;
			&\sigma_2=\begin{array}{lcr}
			\begin{tikzpicture}
			\begin{scope}[scale=.3]
			\permutation{4,1,5,2,3}
			\end{scope}
			\end{tikzpicture}
			\end{array},\;
			\sigma_3=\begin{array}{lcr}
			\begin{tikzpicture}
			\begin{scope}[scale=.3]
			\permutation{5,1,6,3,4,2}
			\end{scope}
			\end{tikzpicture}
			\end{array},\;
			\sigma_4=\begin{array}{lcr}
			\begin{tikzpicture}
			\begin{scope}[scale=.3]
			\permutation{6,1,7,3,4,2,5}
			\end{scope}
			\end{tikzpicture}
			\end{array},\\
			\sigma_5=&\begin{array}{lcr}
			\begin{tikzpicture}
			\begin{scope}[scale=.3]
			\permutation{7,1,8,3,4,2,5,6}
			\end{scope}
			\end{tikzpicture}
			\end{array},\;
			\sigma_6=\begin{array}{lcr}
			\begin{tikzpicture}
			\begin{scope}[scale=.3]
			\permutation{8,1,9,4,5,2,6,7,3}
			\end{scope}
			\end{tikzpicture}
			\end{array}.
			\end{split}
			\end{equation*}\\
			\caption{The diagrams of the six permutations considered in Example \ref{exemp:overlap_graph3}. Note that every permutation is obtained by adding a new final point to the previous one. \label{fig:diag}}
		\end{center}
	\end{figure}
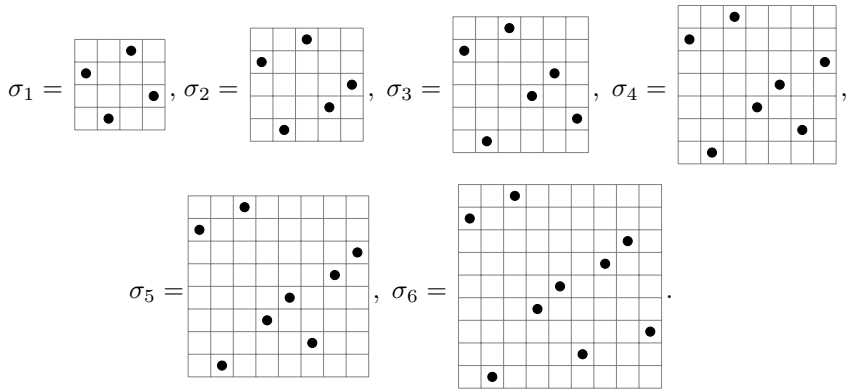
\end{example}

We conclude this section with two simple results useful for the following sections.

\begin{lemma}\label{lemma:pathocc}
	If $\sigma $ is a permutation and $w=W_k(\sigma )=(e_1, \dots , e_s)$ is its corresponding walk on $\ValGraph[k]$, then, for all $\pi\in\SS_k,$
	$$\coc(\pi, \sigma ) =  | \{i \leq s | \lb (e_i )  =  \pi \}|.$$
\end{lemma}

\begin{proof}
	This is a trivial consequence of the definition of the map $W_k.$ See in particular \cref{eq:def_walf}.
\end{proof}

\begin{observation}\label{obs:gluepermutations}
	Let $\pi_1$ and $\pi_2$ be two permutations of size $k - 1 \geq 1$, and take $\tau=\pi_1\oplus\pi_2$.
	Then the path $W_k(\tau)$ goes from $\pi_1$ to $\pi_2$. Consequently, $\ValGraph[k]$ is strongly connected.
\end{observation}

\subsection{A description of the feasible region \texorpdfstring{$P_k$}{the feasible region}}
The goal of this section is to prove the following result.
\begin{theorem}\label{thm:vertices}
The feasible region $P_k$ is the cycle polytope of the overlap graph $\ValGraph[k]$, i.e.\
\begin{equation}
\label{eq:Pn_equal_POn}
	P_k=P(\ValGraph[k]).
\end{equation}
As a consequence, the set of vertices of $P_k$ is $\{\vec{e}_{\mathcal{C}} | \, \mathcal{C} \text{ is a simple cycle of } \ValGraph[k] \}$ and the dimension of $P_k$ is $k! - (k-1)!$. Moreover, the polytope $P_k$ is described by the equations
$$ P_k = \left\{ \vec{v}\in [0,1]^{\SS_k} \Bigg| \sum_{\pi\in \SS_k} v_\pi = 1\, , \, \, \sum_{\be(\pi) = \rho} v_\pi = \sum_{\en (\pi) = \rho} v_\pi \, , \forall \rho \in \SS_{k-1} \right\} \, .$$
\end{theorem}

Before presenting the proof of this theorem, we first give a lemma on directed graphs.

\begin{lemma}
\label{lm:decomp}
	Let $G$ be a directed graph and $w$ a walk on it. Then the multiset of edges of $w$ can be decomposed into $\ell$ simple cycles (for some $\ell\geq 0$) $\mathcal{C}_1, \dots ,  \mathcal{C}_\ell$ and a tail $\mathcal{T}$ that does not repeat vertices (but is possibly empty).
	Specifically, we have the following relation of multisets of edges of $G$:
	$$w = \mathcal{C}_1 \uplus \dots \uplus \mathcal{C}_\ell \uplus \mathcal{T} \, . $$
\end{lemma}

\begin{proof}
This decomposition is obtained inductively on the number of edges.
If $w$ has no repeated vertices, the decomposition $w = \mathcal{T} $ satisfies the desired conditions.
If $w$ has repeated vertices, it has a simple cycle corresponding to the first repetition of a vertex.
By pruning from the walk this simple cycle, we obtain a smaller walk which decomposes by the induction hypothesis.
This gives us the desired result.
\end{proof}

\begin{remark}
The decomposition obtained above, of a walk $w$ into cycles $\mathcal{C}_1, \dots ,  \mathcal{C}_\ell$ and a tail $\mathcal{T}$, is a decomposition of the edge multiset.
In particular, each of the cycles $\mathcal C_i $ or the tail $\mathcal T $ are \emph{not} necessarily formed by consecutive sequences of edges of $w$.
Explicit examples can be readily found.
\end{remark}

\begin{proof}[Proof of Theorem \ref{thm:vertices}]

The first step is to show that, for any simple cycle $\sc $ of $\ValGraph[k]$,  the vector $\vec{e}_{\sc }$ is in $ P_k$. This, together with Proposition \ref{prop:P_n_is_covex} implies that $ P(\ValGraph[k]) \subseteq P_k $.

According to Lemma \ref{lemma:pathperm}, for every $m\in\NN$, there is a permutation $\sigma^m$ such that $W_k(\sigma^m )$ is the walk resulting from the concatenation of $m$ copies of $\sc$.
We claim that $\pcoc_k(\sigma^m) \to \vec{e}_{\sc }$ and $|\sigma^m | \to \infty$. The latter affirmation is trivial since, by Lemma \ref{lemma:pathperm}, $|\sigma^m |= |\sc|m+k-1$.
For the first claim, according to Lemma \ref{lemma:pathocc}, we have that $\coc(\pi, \sigma^m) = m$ for any $\pi $ that is the label of an edge in the simple cycle $\sc $, and 
$\coc(\pi, \sigma^m) = 0$ otherwise.
Hence 
$$\pcoc_k (\sigma^m)  = \vec{e}_{\sc } \frac{m |\sc |}{|\sigma^m|} \to \vec{e}_{\sc } \, , $$ 
as desired.
\medskip

On the other hand, suppose that $\vec{v} \in P_k$, so we have a sequence $\sigma^m$ of permutations such that $|\sigma^m| \to \infty $ and $ \pcoc_k(\sigma^m) \to \vec{v} $.
We will show that $\dist (\pcoc_k(\sigma^m), P(\ValGraph[k]))\to 0$.
It is then immediate, since $P(\ValGraph[k])$ is closed, that $\vec{v} \in P(\ValGraph[k])$, proving that $P_k = P(\ValGraph[k])$.
We consider the walk $w^m=W_k(\sigma^m)$. Using  Lemma \ref{lm:decomp}, the edge multiset of the walk $w^m$ can be decomposed into simple cycles and a tail (that does not repeat vertices and may be empty) as follows
$$w^m= \mathcal{C}^m_1 \uplus  \dots \uplus  \mathcal{C}^m_\ell \uplus  \mathcal{T}^m.$$

Then from Lemma \ref{lemma:pathocc} we can compute $\pcoc_k (\sigma^m)$ as a convex combination of $\vec{e}_{\sc }$ for some simple cycles $\sc$, plus a small error term. Specifically,
\begin{equation*}
\pcoc_k(\sigma^m)=\vec{e}_{\mathcal{C}^m_1}\tfrac{|\mathcal{C}^m_1|}{|\sigma^m|}+\dots+\vec{e}_{\mathcal{C}^m_\ell}\tfrac{|\mathcal{C}^m_\ell|}{|\sigma^m|}+\vec{Er}^m,
\end{equation*}
where $|\vec{Er}^m| \leq\frac{(k-1)!}{|\sigma^m|}$ since there are $(k-1)!$ distinct vertices in $\ValGraph[k]$ and the path $\mathcal{T}^m$ does not contain repeated vertices. In particular, $|\vec{Er}^m|\to 0$ since $k$ is constant and $|\sigma^m|\to\infty$. 

Noting that
$$\pcoc_k(\sigma^m)=\vec{Er}^m+\frac{\sum_i|\mathcal{C}^m_i|}{|\sigma^m|}\vec{w}_m,$$ 
where $\vec{w}_m=\frac{1}{\sum_i|\mathcal{C}^m_i|}\left(\vec{e}_{\mathcal{C}^m_1}|\mathcal{C}^m_1|+\dots+\vec{e}_{\mathcal{C}^m_\ell}|\mathcal{C}^m_\ell|\right)\in P(\ValGraph[k])$, we can conclude that

\begin{equation*}
\dist\left(\pcoc_k(\sigma^m),P(\ValGraph[k])\right) \leq \dist\left(\vec{Er}^m+\frac{\sum_i|\mathcal{C}^m_i|}{|\sigma^m|}\vec{w}_m, \vec{w}_m \right)\to 0 \, ,
\end{equation*}
since $|\vec{Er}^m|\to 0$, $\frac{\sum_i|\mathcal{C}^m_i|}{|\sigma^m|}\to 1$ and $\vec{w}_m$ is uniformly bounded.
This concludes the proof of \cref{eq:Pn_equal_POn}. 

The characterization of the vertices is a trivial consequence of Proposition \ref{prop:vertices}.
For the dimension, it is enough to note that $\ValGraph[k]$ is strongly connected from Observation \ref{obs:gluepermutations}.
So by Proposition \ref{thm:dimension} it has dimension $| E(\ValGraph[k])| - | V( \ValGraph[k])| = k! - (k-1)!$, as desired.
Finally, the equations for $P_k$ are determined using Proposition \ref{thm:faces} and the definition of the overlap graph.
\end{proof}

\begin{remark}\label{rem:simple cycles}
	Since for two different simple cycles $\sc_1, \sc_2$ we have that $ \vec{e}_{\sc_1} \neq \vec{e}_{\sc_2}$, enumerating the vertices corresponds to enumerating simple cycles of $\ValGraph[k]$ (which seems to be a difficult problem).
This problem was partially investigated in \cite{asplund2018enumerating}.
There, all the cycles of size one and two are enumerated.
\end{remark}

\section{Mixing classical patterns and consecutive patterns}\label{sec:class_and_cons}

In \cref{sect:mixing} we explained that a natural question is to describe the feasible region when we mix classical and consecutive patterns.

More generally, let $\mathcal{A},\mathcal{B}\subseteq\SS$ be two finite sets of permutations. We consider the following sets of points
\begin{equation}\label{eq:ABdef}
\begin{split}
	&A = \left\{\vec{v}\in [0,1]^{\mathcal{A}} | \exists\; (\sigma^m)_{m\in\NN}\in{\SS}^{\NN}\text{ s.t. }|\sigma^m| \to \infty\text{ and } \left(\pcoc(\pi, \sigma^m )\right)_{\pi\in\mathcal{A}} \to \vec{v}   \right\},\\
	&B = \left\{\vec{v}\in [0,1]^{\mathcal{B}} | \exists\; (\sigma^m)_{m\in\NN}\in{\SS}^{\NN}\text{ s.t. }|\sigma^m| \to \infty\text{ and } \left(\poc(\pi, \sigma^m )\right)_{\pi\in\mathcal{B}} \to \vec{v}   \right\}.
\end{split}
\end{equation}
We want to investigate the set
\begin{equation}\label{eq:Cdef}
\begin{split}
C=\Big\{\vec{v}\in [0,1]^{\mathcal{A}\sqcup\mathcal{B}} \Big|& \exists\; (\sigma^m)_{m\in\NN}\in{\SS}^{\NN}\text{ s.t. }|\sigma^m| \to \infty,\\
 &\left(\pcoc(\pi, \sigma^m )\right)_{\pi\in\mathcal{A}} \to (\vec{v})_{\mathcal{A}} \text{ and } \left(\poc(\pi, \sigma^m )\right)_{\pi\in\mathcal{B}} \to (\vec{v})_{\mathcal{B}}  \Big\}.
\end{split}
\end{equation}

For the statement of the next theorem we need to recall the definition of the \emph{substitution operation} on permutations.
For $\theta, \nu^{(1)}, \ldots, \nu^{(d)}$ permutations such that $d = |\theta |$, the substitution $\theta[\nu^{(1)},...,\nu^{(d)}]$ is defined as follows: for each $i$, we replace the point $(i,\theta(i))$ in the diagram of $\theta$ with the diagram of $\nu^{(i)}$.
Then, rescaling the rows and columns yields the diagram of a larger permutation $\theta[\nu^{(1)},...,\nu^{(d)}]$.
Note that $|\theta[\nu^{(1)},...,\nu^{(d)}]| = \sum_{i=1}^d |\nu^{(i)}|$ (see \cref{fig:Subs} for an example).
\begin{figure}[ht]
	\includegraphics[height=2.5cm]{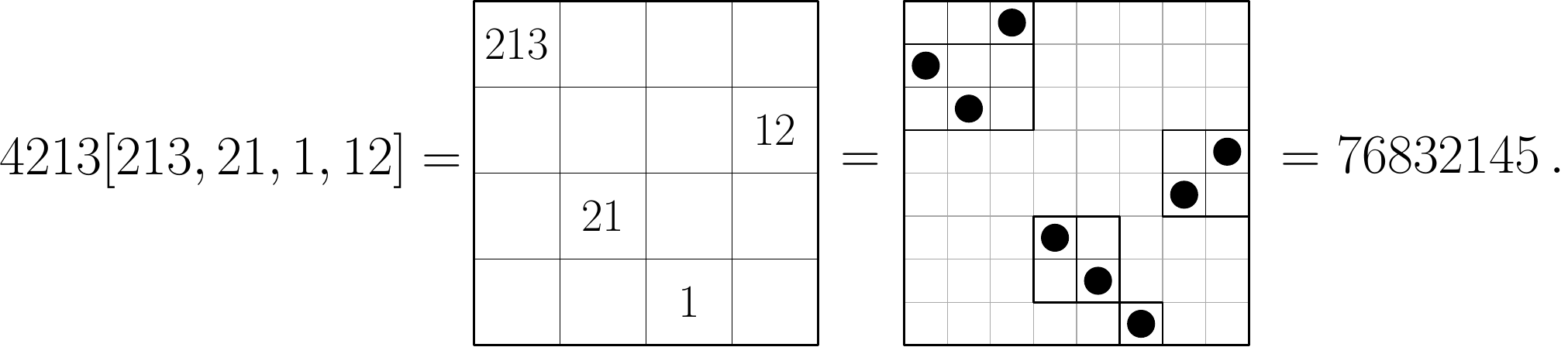}
	\caption{Example of substitution of permutations.}
	\label{fig:Subs}
\end{figure}

\begin{theorem} \label{thm:ABC}
	Let $\mathcal{A}, \mathcal{B}\subseteq \mathcal{S}$ be finite sets of permutations, and $A,B,C$ be defined as in \cref{eq:ABdef,eq:Cdef}. It holds that
	\begin{equation}
	\label{eq:ABC}
			A\times B = C \, .
	\end{equation}
	Specifically, given two points $\vec{v}_{\mathcal{A}} \in A, \vec{v}_{\mathcal{B}} \in B$, consider two sequences
	$(\sigma^m_A)_{m\in\NN}\in{\SS}^{\NN}$ and $(\sigma^m_B)_{m\in\NN}\in{\SS}^{\NN}$ such that 
	\begin{equation}\label{eq:assump}
	\begin{split}
	&|\sigma^m_A| \to \infty\text{ and } \left(\pcoc(\pi, \sigma^m_A)\right)_{\pi\in\mathcal{A}} \to \vec{v}_A,\\
	&|\sigma^m_B| \to \infty\text{ and } \left(\poc(\pi, \sigma^m_B )\right)_{\pi\in\mathcal{B}} \to \vec{v}_B,
	\end{split}
	\end{equation}
	then the sequence $(\sigma^m_C)_{m\in\NN}$ defined by
	\begin{equation}\label{eq:assump2}
	 \sigma^m_C\coloneqq\sigma^m_B[\sigma^m_A,\dots,\sigma^m_A],\quad\text{for all}\quad m\in\NN,
	\end{equation}
	satisfies
\begin{equation}\label{eq:goaloftheproof}
|\sigma^m_C|\to \infty, \quad\left(\pcoc(\pi, \sigma^m_C)\right)_{\pi\in\mathcal{A}} \to \vec{v}_A \quad\text{and}\quad \left(\poc(\pi, \sigma^m_C )\right)_{\pi\in\mathcal{B}} \to \vec{v}_B.
\end{equation}

\end{theorem}

\begin{proof}
Let $(\sigma^m_C)_{m\in\NN}$ be defined as in \cref{eq:assump2}.
The fact that the size of $\sigma^m_C$ tends to infinity follows from $|\sigma^m_C|=|\sigma^m_A||\sigma^m_B| \to \infty $.
For the second limit in \cref{eq:goaloftheproof}, note that for every pattern $\pi\in\mathcal{A},$
$$\pcoc(\pi, \sigma^m_C)=\frac{\coc(\pi, \sigma^m_C)}{|\sigma^m_B|\cdot|\sigma^m_A|}=\frac{\coc(\pi,\sigma^m_A)\cdot|\sigma^m_B|+Er}{|\sigma^m_B|\cdot|\sigma^m_A|} = \pcoc(\pi,\sigma^m_A) + \frac{Er}{|\sigma^m_B|\cdot|\sigma^m_A|} ,$$
where $Er\leq|\sigma^m_B|\cdot |\pi|$. 
This error term comes from intervals of $[|\sigma^m_C |]$ that intersect more than one copy of $\sigma^m_A$. Since $|\pi|$ is fixed and $|\sigma^m_A| \to \infty$ we can conclude the desired limit, using the assumption in \cref{eq:assump} that $\left(\pcoc(\pi, \sigma^m_A)\right)_{\pi\in\mathcal{A}} \to \vec{v}_A$.

Finally, for the third limit in \cref{eq:goaloftheproof} we note that setting $n=|\sigma^m_C|$ and $k=|\pi|$,
\begin{equation}
\label{eq:step_1}
	\poc(\pi,\sigma^m_C) = \frac{\occ(\pi,\sigma^m_C)}{\binom{n}{k}} \, 
	=\mathbb{P} \left(\pat_{\bm I}(\sigma^m_C)=\pi \right),
\end{equation}
where $\bm I$ is a random set, uniformly chosen among the $\binom{n}{k}$ subsets of $[n]$ with $k$ elements (we denote random quantities in \textbf{bold}). Let now $E^m$ be the event that the random set $\bm I$ contains two indices $\bm i,\bm j$ of $[|\sigma^m_C|]$ that belong to the same copy of $\sigma^m_A$ in $\sigma^m_C$. Denote by $(E^m)^C$ the complement of the event $E^m$. We have
\begin{multline}
\label{eq:step_2}
	\mathbb{P} \left(\pat_{\bm I}(\sigma^m_C)=\pi \right)=\\
	\mathbb{P} \left(\pat_{\bm I}(\sigma^m_C)=\pi |E^m\right)\cdot\mathbb{P}\left(E^m\right)+\mathbb{P} \left(\pat_{\bm I}(\sigma^m_C)=\pi |(E^m)^C\right)\cdot\mathbb{P}\left((E^m)^C\right).
\end{multline}
We claim that
\begin{equation}
\label{eq:step_3}
	\mathbb{P}\left(E^m\right)\leq\binom{k}{2}\frac{1}{|\sigma^m_B|}\to0.
\end{equation}
Indeed, the factor $\binom{k}{2}$ counts the number of pairs $i,j$ in a set of cardinality $k$ and the factor $\frac{1}{|\sigma^m_B|}$ is an upper bound for the probability that given a uniform two-element set $\left\{\bm i,\bm j\right\}$ then $\bm i, \bm j$ belong to the same copy of $\sigma^m_A$ in $\sigma^m_C$ (recall that there are $|\sigma^m_B|$ copies of $\sigma^m_A$ in $\sigma^m_C$). Note also that 
\begin{equation}
\label{eq:step_4}
\mathbb{P} \left(\pat_{\bm I}(\sigma^m_C)=\pi |(E^m)^C\right)=\poc(\pi,\sigma^m_B)\to \vec{v}_B,
\end{equation}
where the last limit comes from \cref{eq:assump}. Using \cref{eq:step_1,eq:step_2,eq:step_3,eq:step_4}, we obtain that $$\left(\poc(\pi, \sigma^m_C )\right)_{\pi\in\mathcal{B}} \to \vec{v}_B.$$ This concludes the proof of \cref{eq:goaloftheproof}.
The result in \cref{eq:ABC} follows from the fact that we trivially have $  C \subseteq A \times B$, and for the other inclusion we use the construction above, which proves that $(\vec{v}_A,\vec{v}_B)\in C,$ for every $\vec{v}_{\mathcal{A}} \in A, \vec{v}_{\mathcal{B}} \in B$.
\end{proof}

\longthanks{The authors are very grateful to Mathilde Bouvel and Valentin F\'eray for various discussions during the preparation of the paper. They also thank the anonymous referees for all their precious and useful comments.}

\bibliographystyle{amsplain-ac}
\bibliography{bibli}

\end{document}